\documentclass{amsart}
\usepackage[toc,page]{appendix}
\usepackage{etex}
\usepackage[utf8]{inputenc}
\usepackage[T1]{fontenc}
\usepackage{lmodern}
\usepackage{amscd,amssymb,microtype,mathtools,hyperref}
\usepackage[all,cmtip]{xy}
\usepackage{tikz-cd}
\newtheorem{thm}{Theorem}[section]
\newtheorem{theorem}[thm]{Theorem}
\newtheorem{lem}[thm]{Lemma}
\newtheorem{lemma}[thm]{Lemma}

\newtheorem{prop}[thm]{Proposition}
\newtheorem{proposition}[thm]{Proposition}

\theoremstyle{remark}

\newtheorem{situation}[thm]{Situation}
\newtheorem{remark}[thm]{Remark}

\theoremstyle{definition}

\newtheorem{definition}[thm]{Definition}

\DeclareMathOperator{\chari}{char}

\numberwithin{equation}{thm}

\allowdisplaybreaks[4]

\begin{document}

\vfuzz0.5pc
\hfuzz0.5pc % Don't bother to report
          % overfull boxes if overage is < 6pt

\newcommand{\claimref}[1]{Claim~\ref{#1}}
\newcommand{\thmref}[1]{Theorem~\ref{#1}}
\newcommand{\propref}[1]{Proposition~\ref{#1}}
\newcommand{\lemref}[1]{Lemma~\ref{#1}}
\newcommand{\coref}[1]{Corollary~\ref{#1}}
\newcommand{\remref}[1]{Remark~\ref{#1}}
\newcommand{\conjref}[1]{Conjecture~\ref{#1}}
\newcommand{\questionref}[1]{Question~\ref{#1}}
\newcommand{\defnref}[1]{Definition~\ref{#1}}
\newcommand{\secref}[1]{\S\ref{#1}}
\newcommand{\ssecref}[1]{\ref{#1}}
\newcommand{\sssecref}[1]{\ref{#1}}

\newcommand{\RED}{{\mathrm{red}}}
\newcommand{\tors}{{\mathrm{tors}}}
\newcommand{\eq}{\Leftrightarrow}

\newcommand{\mapright}[1]{\smash{\mathop{\longrightarrow}\limits^{#1}}}
\newcommand{\mapleft}[1]{\smash{\mathop{\longleftarrow}\limits^{#1}}}
\newcommand{\mapdown}[1]{\Big\downarrow\rlap{$\vcenter{\hbox{$\scriptstyle#1$}}$}}
\newcommand{\smapdown}[1]{\downarrow\rlap{$\vcenter{\hbox{$\scriptstyle#1$}}$}}

\newcommand{\I}{\mathcal{I}}
\newcommand{\J}{\mathcal{J}}
\newcommand{\CO}{\mathcal{O}}
\newcommand{\C}{\mathcal{C}}
\newcommand{\BC}{\mathbb{C}}
\newcommand{\BQ}{\mathbb{Q}}
\newcommand{\m}{\mathcal{M}}
\newcommand{\h}{\mathcal{H}}
\newcommand{\Z}{\mathcal{Z}}
\newcommand{\BZ}{\mathbb{Z}}
\newcommand{\W}{\mathcal{W}}
\newcommand{\Y}{\mathcal{Y}}
\newcommand{\T}{\mathcal{T}}
\newcommand{\BP}{\mathbb{P}}
\newcommand{\CP}{\mathcal{P}}
\newcommand{\G}{\mathbb{G}}
\newcommand{\BR}{\mathbb{R}}
\newcommand{\D}{\mathcal{D}}
\newcommand{\DD}{\mathcal{D}}
\newcommand{\LL}{\mathcal{L}}
\newcommand{\f}{\mathcal{F}}
\newcommand{\E}{\mathcal{E}}
\newcommand{\BN}{\mathbb{N}}
\newcommand{\N}{\mathcal{N}}
\newcommand{\K}{\mathcal{K}}
\newcommand{\R} {\mathbb{R}}
\newcommand{\PP}{\mathbb{P}}
\newcommand{\Pp}{\mathbb{P}}
\newcommand{\BF}{\mathbb{F}}
\newcommand{\QQ}{\mathcal{Q}}
\newcommand{\closure}[1]{\overline{#1}}
\newcommand{\EQ}{\Leftrightarrow}
\newcommand{\imply}{\Rightarrow}
\newcommand{\isom}{\cong}
\newcommand{\embed}{\hookrightarrow}
\newcommand{\tensor}{\mathop{\otimes}}
\newcommand{\wt}[1]{{\widetilde{#1}}}
\newcommand{\ol}{\overline}
\newcommand{\ul}{\underline}
\newcommand{\et}{\textup{\'et}}

\newcommand{\bs}{{\backslash}}
\newcommand{\CS}{\mathcal{S}}
\newcommand{\CA}{\mathcal{A}}
\newcommand{\Q}{\mathbb{Q}}
\newcommand{\F}{\mathcal{F}}
\newcommand{\sing}{{\text{sing}}}
\newcommand{\B}{\mathcal{B}}
\newcommand{\X}{\mathcal{X}}
\newcommand{\EE}{\mathbb{E}}

% Javier-Macro
\newcommand{\ECS}[1]{E_{#1} (X)}
\newcommand{\CV}[2]{{\mathcal{C}}_{#1,#2} ({X})}
\newcommand{\LC}{\mathbb{L}}
\newcommand{\OO}{\mathcal{O}}
\newcommand{\rank}{\mathop{\mathrm{rank}}\nolimits}
\newcommand{\codim}{\mathop{\mathrm{codim}}\nolimits}
\newcommand{\Ord}{\mathop{\mathrm{Ord}}\nolimits}
\newcommand{\Var}{\mathop{\mathrm{Var}}\nolimits}
\newcommand{\Ext}{\mathop{\mathrm{Ext}}\nolimits}
\newcommand{\EXT}{\mathop{\mathcal{E}\mathrm{xt}}\nolimits}
\newcommand{\Pic}{\mathop{\mathrm{Pic}}\nolimits}
\newcommand{\Spec}{\mathop{\mathrm{Spec}}\nolimits}
\newcommand{\Jac}{\mathop{\mathrm{Jac}}\nolimits}
\newcommand{\Div}{\mathop{\mathrm{Div}}\nolimits}
\newcommand{\sgn}{\mathop{\mathrm{sgn}}\nolimits}
\newcommand{\supp}{\mathop{\mathrm{supp}}\nolimits}
\newcommand{\Hom}{\mathop{\mathrm{Hom}}\nolimits}
\newcommand{\Sym}{\mathop{\mathrm{Sym}}\nolimits}
\newcommand{\nilrad}{\mathop{\mathrm{nilrad}}\nolimits}
\newcommand{\Ann}{\mathop{\mathrm{Ann}}\nolimits}
\newcommand{\Proj}{\mathop{\mathrm{Proj}}\nolimits}
\newcommand{\mult}{\mathop{\mathrm{mult}}\nolimits}
\newcommand{\Bs}{\mathop{\mathrm{Bs}}\nolimits}
\newcommand{\Span}{\mathop{\mathrm{Span}}\nolimits}
\newcommand{\IM}{\mathop{\mathrm{Im}}\nolimits}
\newcommand{\Hol}{\mathop{\mathrm{Hol}}\nolimits}
\newcommand{\End}{\mathop{\mathrm{End}}\nolimits}
\newcommand{\CH}{\mathop{\mathrm{CH}}\nolimits}
\newcommand{\Exec}{\mathop{\mathrm{Exec}}\nolimits}
\newcommand{\SPAN}{\mathop{\mathrm{span}}\nolimits}
\newcommand{\birat}{\mathop{\mathrm{birat}}\nolimits}
\newcommand{\cl}{\mathop{\mathrm{cl}}\nolimits}
\newcommand{\rat}{\mathop{\mathrm{rat}}\nolimits}
\newcommand{\Bir}{\mathop{\mathrm{Bir}}\nolimits}
\newcommand{\Rat}{\mathop{\mathrm{Rat}}\nolimits}
\newcommand{\aut}{\mathop{\mathrm{aut}}\nolimits}
\newcommand{\eff}{\mathop{\mathrm{eff}}\nolimits}
\newcommand{\nef}{\mathop{\mathrm{nef}}\nolimits}
\newcommand{\amp}{\mathop{\mathrm{amp}}\nolimits}
\newcommand{\DIV}{\mathop{\mathrm{Div}}\nolimits}
\newcommand{\Bl}{\mathop{\mathrm{Bl}}\nolimits}
\newcommand{\Cox}{\mathop{\mathrm{Cox}}\nolimits}
\newcommand{\NE}{\mathop{\mathrm{NE}}\nolimits}
\newcommand{\NM}{\mathop{\mathrm{NM}}\nolimits}
\newcommand{\Gal}{\mathop{\mathrm{Gal}}\nolimits}
\newcommand{\coker}{\mathop{\mathrm{coker}}\nolimits}
\newcommand{\ch}{\mathop{\mathrm{ch}}\nolimits}
\newcommand{\Char}{\mathop{\mathrm{char}}\nolimits}

\setcounter{tocdepth}{1}
\title[Automorphism and Cohomology: Complete Intersections]{Automorphism and
  Cohomology II:\ Complete intersections}

\author[X.~Chen]{Xi Chen${}^{\dagger}$}
\address{632 Central Academic Building\\
University of Alberta\\
Edmonton, Alberta T6G 2G1, CANADA}
\email{xichen@math.ualberta.ca}

\author[X.~Pan]{Xuanyu Pan} \address{Institute of Mathematics, AMSS, Chinese
  Academy of Sciences, 55 ZhongGuanCun East Road, Beijing 100190, China} \email{pan@amss.ac.cn}

\author[D.~Zhang]{Dingxin Zhang}
\address{YMSC, Tsinghua University \\ 30 Shuangqing Rd, Beijing 100190, China}
\email{zhangdingxin13@gmail.com}

%%\date{February 13, 2015}
% % \date{\today}
% %
\date{October 6, 2019}
\thanks{${}^{\dagger}$ Research partially supported by NSERC 262265.}
% %
% %
% %
\begin{abstract}
We prove that the automorphism group of a general complete intersection $X$ in $\BC\PP^n$ is trivial with a few well-understood
exceptions. We also prove that the automorphism group of a complete intersection
$X$ acts on the cohomology of $X$ faithfully with a few well-understood
exceptions.
\end{abstract}
% %

\maketitle
\tableofcontents

\section{Introduction}

In this paper we prove two results concerning automorphisms of complete
intersections in projective spaces: one on their ``generic triviality'' and the other on the faithfulness of their actions on cohomology groups.

The earliest result in the first direction, to the best of our knowledge, is due to H. Matsumura and P. Monsky~\cite{M-M}:

\begin{thm}[Matsumura and Monsky]\label{THMM-M}
Let $X$ be a smooth hypersurface in $\PP^n$ of degree $d$ over a field $k$. Then
\begin{itemize}
\item $\mathrm{Aut}(X)$ is finite if $n\ge 3$, $d\ge 3$ and $(n,d) \ne (3,4)$;
\item $\mathrm{Aut}(X) = \{1\}$ if $X$ is generic, $n\ge 3$ and $d\ge 3$ except the case
that $(n,d) = (3,4)$ and $\chari(k) > 0$.
\end{itemize}
\end{thm}

Naturally, one may ask whether this result generalizes to complete intersections.
Let us first make the following definition.

\begin{definition}
Let $k$ be a field and $X$ be a closed subscheme of $\mathbb{P}^n_k$. The
linear automorphism group of $X$ with respect to the projective embedding $X
\subset \mathbb{P}^n_k$, denoted by $\mathrm{Aut}_L(X)$, is the subgroup of
$\mathbf{PGL}_{n+1}(k)$ whose linear action on $\mathbb{P}^n_k$ takes $X$ onto
$X$ itself.
\end{definition}

The first part of Theorem~\ref{THMM-M} was generalized to complete intersections by
O.~Benoist~\cite{MR3022710}. He proved that
$\mathrm{Aut}_L(X)$ is finite for a smooth complete intersection
$X\subset \mathbb{P}^n$ with some well-known exceptions. In this paper we study
the automorphism group of a generic complete intersection.

\begin{thm}\label{THMCIAUTO000}
Let $X$ be a smooth complete intersection in $\mathbb{P}_k^n$ of type
$(d_1,d_2,\ldots,d_c)$ over an algebraically closed field $k$, i.e., $X = X_1\cap X_2\cap \ldots\cap X_c$ with
$\deg X_i =d_i$. If $\dim X \ge 2$, $\deg X \ge 3$ and
$2\le d_1\le d_2\le \ldots\le d_c$, then
\begin{enumerate}
\item\label{THMCIAUTO000:1}
$\mathrm{Aut}_L(X) = \{1\}$ for $\chari(k) = 0$
and $|\mathrm{Aut}_L(X)| = (\chari(k))^r$ for $\chari(k)>0$ and some $r\in \BN$ when $X$ is a general complete intersection of type $(d_1,d_2,\ldots,d_c)\ne (2,2)$, and
\item\label{THMCIAUTO000:2}
$\mathrm{Aut}_L(X) = (\BZ/2\BZ){}^{n}$ for $X$ general of type $(2,2)$ and $\chari(k)\ne 2$.
\end{enumerate}
\end{thm}

\begin{remark}
Retain the hypotheses and notation of Theorem~\ref{THMCIAUTO000}
and suppose that $X$ is one of the following smooth complete intersections in $\BC\PP^n$:
\begin{itemize}
\item $\dim X \ge 3$, or
\item $\dim X \ge 2$ and $\omega_X\ne \CO_X$, or
\item $\dim X = 2$, $\omega_X\cong \mathcal{O}_X$ and $\rank_\BZ \Pic(X) = 1$, i.e., $X$ is a K3 surface of Picard rank $1$.
\end{itemize}
Then it is known that $\mathrm{Aut}(X) = \mathrm{Aut}_L(X)$. Thus, by
Theorem~\ref{THMCIAUTO000}, when $X$ is general in its moduli, the
automorphism group of $X$ is trivial.

The triviality of $\mathrm{Aut}_L(X)$ for a general complete intersection $X\subset \BC\mathbb{P}^n$, to the best of our knowledge, was only known in some special cases:
\begin{itemize}
\item
aforementioned hypersurface case~\cite{M-M};
\item
$3 \le d_1 < d_2 \le ... \le d_c$~\cite[Lemma 2.12]{jnl};
\item
$(d_1,d_2,...,d_c) = (2,2,2)$~\cite[Lemma 2.13]{jnl}.
\end{itemize}
For a very general complete intersection $X$ of Calabi-Yau or general type in $\BC\PP^n$,
we have a stronger statement: there are no non-trivial dominant rational self maps
$\sigma: X\dashrightarrow X$~\cite{Che13}.
\end{remark}

Our second result, which is derived from Theorem~\ref{THMCIAUTO000} above and
the method in the papers~\cite{Pan} and~\cite{PANL} of the second author, is
about the faithfulness of the action of \(\mathrm{Aut}(X)\) on its étale
cohomology group. Questions about this faithfulness was explored for varieties
of low dimension, and Burns, Rapoport, Shafarevich, Ogus \ldots confirm this
question for K3 surfaces. Few high-dimensional varieties are known to have a
positive answer to this question.

\begin{thm}\label{autocoh}
Let $k$ be an algebraically closed field and $X$ be a smooth complete
intersection in $\mathbb{P}^n_{k}$ of type $(d_1,d_2,\ldots,d_c)$. Suppose that $\deg X \ge 3$ and
$2\le d_1\le d_2\le \ldots\le d_c$. Let $\ell$ be a prime different from
$\mathrm{char}(k)$.

\begin{enumerate}
\item If $(d_1,d_2,\ldots,d_c)=(2,2)$, $\chari(k)\neq 2$ and $\dim(X)=m(>1)$ is odd
(resp.~even), then the kernel of
\begin{equation*}
\mathrm{Aut}(X)\rightarrow \mathrm{Aut}(\mathrm{H}_{\et}^m(X,\mathbb{Q}_\ell))
\end{equation*}
is $(\mathbb{Z}/2\mathbb{Z}){}^{m+1}$ (resp.~trivial) for X general.

\item\label{item:1} Suppose that $X$ is not an elliptic curve and is not a
  quadric hypersurface. If $(d_1,d_2,\ldots,d_c)\ne (2,2)$,
then the map
\begin{equation}
\label{eq:inj}
\mathrm{Aut}(X)\rightarrow \mathrm{Aut}(\mathrm{H}_{\et}^m(X,\mathbb{Q}_\ell))
\end{equation}
is injective.
\end{enumerate}
\end{thm}

Let us remark that the validity of
item~\ref{autocoh} (\ref{item:1}) relies on certain infinitesimal Torelli results for
complete intersections, i.e., the injectivity of the cup product map
\begin{equation}\label{eq:1}
\mathrm{H}^1(X,T_X)\rightarrow
\bigoplus\limits_{p+q=m}\mathrm{Hom}(\mathrm{H}^q(X,\Omega^p_X),\mathrm{H}^{q+1}(X,\Omega^{p-1}_X)).
\end{equation}
The injectivity indeed holds true for all smooth complete intersections except some trivial cases when the characteristic
of the ground field $k$ is zero (result of~\cite{Inf}). In fact, Flenner's idea also
generalizes to an algebraically closed field without assuming $k = \mathbb{C}$. We
include a proof of this result in an appendix to the paper.

Let us briefly describe the structure of this paper.
% Section~\ref{sec:finit-autom-groups} we give a sketch of the proof the
% finiteness of $\mathrm{Aut}_L(X)$ by showing that $H^0(T_X) = 0$.

In Section~\ref{sec:equiv-koda-spenc}, we prove the generic triviality of
$\mathrm{Aut}_L(X)$ by studying the action of $\mathrm{Aut}(X)$ on $\mathrm{H}^1(T_X)$.
In Section~\ref{sec:equiv-deform}, we use equivariant deformation theory to
derived Theorem~\ref{autocoh}(2) from Theorem~\ref{THMCIAUTO000} when
\(k=\mathbb{C}\). The last section proves Theorem~\ref{autocoh}(1) and uses the
theory in~\cite{PANL} to finish the proof of Theorem~\ref{autocoh}(2) for an
arbitrary field.

\medskip\noindent%
\textbf{Acknowledgments.} The second author thanks Prof.~Johan de Jong for
proposing the questions about automorphism and cohomology in a 2014 summer
school at Seattle. We thank A. Javanpeykar and D. Loughran for some email
exchange explaining their work to us.

\section{Equivariant Kodaira Spencer Map}
\label{sec:equiv-koda-spenc}

In this section, we will prove the triviality of $\mathrm{Aut}_L(X)$
for generic complete intersections $X\subset \PP_k^n$. This is achieved
by studying the natural action $\mathrm{Aut}(X)$ on $\mathrm{H}^1(T_X)$.
We make the following simple observation on this action: if $\sigma\in \mathrm{Aut}(X)$
can be ``deformed'' as $X$ deforms in a family, then $\sigma$ preserves
the Kodaira-Spencer classes of this family at $X$.

\begin{prop}\label{PROPCIAUTO000}
Let $\pi: X\to B$ be a flat family of projective varieties over a smooth
variety $B$ with $\pi$ smooth and proper, where $X$ and $B$ are varieties over an algebraically closed field. Then for
a very general point $b\in B$, $\mathrm{Aut}(X_b)$ acts trivially on the image of the Kodaira-Spencer map $\kappa: T_{B,b} \to \mathrm{H}^1(T_{X_b})$. If we fix a line bundle $L$ of $X$ relatively ample over $B$, then for a general point $b\in B$, $\mathrm{Aut}_L(X_b)$ acts trivially on
$\kappa(T_{B,b})$.
\end{prop}

\begin{proof}
Note that $\mathrm{Aut}(X_b)$ is a locally Noetherian scheme.
At a very general point $b$,
every $\sigma_b\in \mathrm{Aut}(X_b)$
can be ``deformed'' over $B$. That is, after shrinking $B$ and a base change
unramified over $b$,
there exists a $\sigma \in \mathrm{Aut}(X/B)$ such that $\sigma_b$ is the restriction
of $\sigma$ to $b$,
where $\mathrm{Aut}(X/B)$ is the automorphism group of $X$ preserving $B$.
This gives us the commutative diagram
\begin{equation}\label{ECIAUTO001}
\xymatrix{
X\ar[r]^{\sigma} \ar[dr]_{\pi} & X\ar[d]^\pi\\
& B
}
\end{equation}
which leads to the commuting exact sequences
\begin{equation}\label{ECIAUTO021}
\xymatrix{
0 \ar[r] & T_{X_b} \ar[d]^{\sigma_{b,*}}\ar[r] & T_X\Big|_{X_b} \ar[d]^-{\sigma_*} \ar[r] & N_{X_b/X}
\ar@{=}[d]\ar[r] & 0\\
0 \ar[r] & T_{X_b} \ar[r] & T_X\Big|_{X_b} \ar[r] & N_{X_b/X} \ar[r] & 0
}
\end{equation}
where $N_{X_b/X} \isom \pi^* T_{B,b}$ is the normal bundle of $X_b$ in $X$. Consequently, the resulting Kodaira-Spencer map
commutes with the action of $\sigma_b$, i.e.,
the diagram
\begin{equation}\label{ECIAUTO022}
\xymatrix{
T_{B,b} \ar@{=}[d] \ar[r]^-{\kappa} & \mathrm{H}^1(T_{X_b})\ar[d]^-{\sigma_{b,*}}\\
T_{B,b} \ar[r]^-{\kappa} & \mathrm{H}^1(T_{X_b})
}
\end{equation}
commutes. Therefore, the image of $\kappa$ is fixed under the action of $\sigma_b$.

If we fix a polarization $L$, $\mathrm{Aut}_L(X_b)$ is a scheme. Therefore,
at a general point $b$, every $\sigma_b\in \mathrm{Aut}_L(X_b)$
can be ``deformed'' over $B$. Then the above argument shows that
$\kappa(T_{B,b})$ is fixed under the action of $\sigma_b$.
\end{proof}

As a complete intersection $X\subset \PP^n$ of type $(d_1,d_2,\ldots,d_c)$ deforms
in $\PP^n$, the image of the Kodaira-Spencer map is exactly the cokernel $\coker(J_X)$ of $J_X$ given in the diagram
\begin{equation}\label{ECIAUTO000}
\xymatrix{
& & \mathrm{H}^0(\CO_X(1))^{\oplus n+1} \ar[d] \ar[r]^-{J_X} &
\displaystyle{\bigoplus_{i=1}^c} \mathrm{H}^0(\CO_X(d_i)) \ar@{=}[d]\\
0 \ar[r] & \mathrm{H}^0(T_{X}) \ar[r] & \mathrm{H}^0(T_{\PP^n}\Big|_X) \ar[r] &
\mathrm{H}^0(N_{X/\PP^n})
}
\end{equation}
where $\mathrm{H}^0(\CO_X(1))^{\oplus n+1}\to \mathrm{H}^0(X, T_{\PP^n})$ is induced by
the Euler sequence and $N_{X/\PP^n}$ is the normal bundle of $X\subset\PP^n$.
Thus, by \propref{PROPCIAUTO000}, $\mathrm{Aut}_L(X)$ acts trivially on $\coker(J_X)$
for $X$ general.

Since Benoist~\cite{MR3022710} has shown that $\mathrm{Aut}_L(X)$ is finite under the hypothesis of
Theorem~\ref{THMCIAUTO000}, every
$\sigma\in \mathrm{Aut}_L(X)$ has finite order. Suppose that $\mathrm{Aut}_L(X) \ne \{1\}$ and $|\mathrm{Aut}_L(X)| \ne (\chari(k))^r$ if $\chari(k)>0$.
Then there exists $\sigma\in \mathrm{Aut}_L(X)$ of order $p$ for some prime $p\ne \chari(k)$.
That is, $\sigma^p = 1$ and $\sigma \ne 1$.

Since $\sigma\in \mathrm{Aut}(\PP^n)$, it has a dual action on $\PP \mathrm{H}^0(\CO_{\mathbb{P}^{n}}(1))$ and hence on $\PP \mathrm{H}^0(\CO_{\mathbb{P}^{n}}(m))$ for all $m\in \BZ^+$.
Since $p\ne \chari(k)$, the lifts of $\sigma$ to $\mathrm{H}^0(\CO_{\mathbb{P}^{n}}(m))$ are diagonalizable. We may choose
a lift of $\sigma$ to $\mathrm{H}^0(\CO_{\mathbb{P}^{n}}(1))$ such that
\begin{equation}\label{ECIAUTO023}
\mathrm{H}^0(\CO_{\mathbb{P}^{n}}(m)) = \bigoplus E_{m,\xi}
\end{equation}
where $E_{m,\xi}$ is the eigenspace of $\sigma$ acting on $\mathrm{H}^0(\CO_{\mathbb{P}^{n}}(m))$
corresponding to the eigenvalue $\xi$ satisfying
$\xi^p = 1$.

Since $\sigma\in \mathrm{Aut}_L(X)$, $\sigma\in \mathrm{Aut}_L(X_1\cap X_2
\cap \ldots \cap X_l)$ for all $d_l < d_{l+1}$. Therefore,
$\sigma$ also acts on the subspace
$\mathrm{H}^0(I_{X_1\cap X_2\cap \ldots\cap X_l}(m))\subset \mathrm{H}^0(\CO_{\mathbb{P}^{n}}(m))$ and
this action is also diagonalizable, where $I_{X_1\cap X_2\cap \ldots\cap X_l}$
is the ideal sheaf of $X_1\cap X_2\cap \ldots\cap X_l$ in $\PP^n$.
Thus, we can choose the defining equations $F_1, F_2, \ldots, F_c$
of $X_1, X_2, \ldots, X_c$ such that
\begin{equation}\label{ECIAUTO024}
\begin{bmatrix}
 \sigma(F_1)\\
 \sigma(F_2)\\
 \vdots\\
 \sigma(F_c)
\end{bmatrix}
= \begin{bmatrix}
   \xi_1\\
   &\xi_2\\
   && \ddots\\
   &&& \xi_c
  \end{bmatrix}
\begin{bmatrix}
 F_1\\
 F_2\\
 \vdots\\
 F_c
\end{bmatrix}
\end{equation}
for some $\xi_i^p = 1$. That is, $F_i\in E_{d_i, \xi_i}$
for $i=1,2,\ldots,c$.

Note that $J_X$ is explicitly given by
\begin{equation}\label{ECIAUTO100}
J_{X}\begin{bmatrix}
s_0\\
s_1\\
\vdots\\
s_n
\end{bmatrix}
= \begin{bmatrix}
\displaystyle{\frac{\partial F_i}{\partial z_j}}
\end{bmatrix}_{c\times (n+1)}
\begin{bmatrix}
s_0\\
s_1\\
\vdots\\
s_n
\end{bmatrix}
\end{equation}
for $(z_0,z_1,...,z_n)$ the homogeneous coordinates of $\mathbb{P}^n$.

Under our choice of the defining
equations of $X$, the action of $\sigma$ on $\coker(J_X)$ is induced by
\begin{equation}\label{ECIAUTO025}
\begin{bmatrix}
G_1\\
G_2\\
\vdots\\
G_c
\end{bmatrix}^\sigma
= \begin{bmatrix}
\xi_1^{-1}\\
&\xi_2^{-1}\\
&& \ddots\\
&&& \xi_c^{-1}
\end{bmatrix}
\begin{bmatrix}
\sigma(G_1)\\
\sigma(G_2)\\
\vdots\\
\sigma(G_c)
\end{bmatrix}
\end{equation}
for $G_i\in \mathrm{H}^0(\CO(d_i))$.
This comes from the observation that $\sigma$ sends
an infinitesimal deformation
$\{ F_i + t_i G_i = 0\}$ of $X$ to
\[
\{ \sigma(F_i + t_i G_i) = 0\} =
\{F_i + t_i \xi_i^{-1} \sigma(G_i) = 0\}.
\]
The fact that
$\sigma$ acts trivially on $\coker(J_X)$ is equivalent to saying that
\begin{equation}\label{ECIAUTO026}
\begin{bmatrix}
\sigma(G_1) - \xi_1 G_1\\
\sigma(G_2) - \xi_2 G_2\\
\vdots\\
\sigma(G_c) - \xi_c G_c
\end{bmatrix}
\in \IM J_X
\end{equation}
for all $G_i \in \mathrm{H}^0(\CO_{\mathbb{P}^{n}}(d_i))$. In other words,
\begin{equation}\label{ECIAUTO027}
\begin{aligned}
&\quad \IM J_X + \left(E_{X, d_1,\xi_1} \oplus E_{X, d_2,\xi_2} \oplus \ldots \oplus E_{X, d_c,\xi_c}
\right)\\
&= \mathrm{H}^0(\CO_X(d_1))\oplus \mathrm{H}^0(\CO_X(d_2)) \oplus \ldots \oplus \mathrm{H}^0(\CO_X(d_c)),
\end{aligned}
\end{equation}
where $\IM(J_X)$ is the image of $J_X$ and
$E_{X,d,\xi}$ is the restriction of $E_{d,\xi}\subset\mathrm{H}^0(\CO_{\mathbb{P}^{n}}(d))$ to $\mathrm{H}^0(\CO_X(d))$.
We will show that this cannot hold by a dimension count. That is,
we are going to show that
\begin{equation}\label{ECIAUTO040}
\begin{aligned}
&\quad \dim \left(\IM J_X + \left( E_{X,d_1,\xi_1} \oplus E_{X,d_2,\xi_2}
\oplus \ldots \oplus E_{X, d_c,\xi_c}
\right)\right)\\
&< h^0(\CO_X(d_1)) + h^0(\CO_X(d_2)) + ... + h^0(\CO_X(d_c)).
\end{aligned}
\end{equation}

Let
\begin{equation}\label{ECIAUTO028}
a_j = \dim E_{1,\eta_j}
\end{equation}
where $\eta_0, \eta_1, \ldots,\eta_{p-1}$ are the $p$-th roots of unit.
Note that $\sum a_j = n+1$ and at least two $a_j$'s are positive, i.e., $\mu\ge 2$ for
\begin{equation}\label{ECIAUTO119}
\mu = \# \left\{ j : a_j > 0, 0\le j\le p-1\right\}
\end{equation}
since $\sigma\ne 1$.
Our argument for~\eqref{ECIAUTO040} is based on the following inequalities:

\begin{lem}\label{LEMCIAUTO100}
Let $X$, $J_X$, $\sigma$, $\xi_i$, $E_{X,d_i,\xi_i}$ and $a_j$ be given as above. Then
\begin{equation}\label{ECIAUTO030}
\begin{split}
&\quad \dim \IM J_{X} - \dim \left(\IM J_{X}\cap
\left(E_{X, d_1,\xi_1} \oplus E_{X, d_2,\xi_2} \oplus \ldots \oplus E_{X, d_c,\xi_c}
\right)\right)\\
&\le (n+1){}^2 - \sum_{j=0}^{p-1} a_j^2.
\end{split}
\end{equation}
\end{lem}

\begin{proof}[Proof of Lemma \ref{LEMCIAUTO100}]
We choose homogeneous coordinates $(z_0, z_1, ..., z_n)$ of $\mathbb{P}^n$ such that
$\sigma(z_i) = \lambda_i z_i$ for each $i$, where $\lambda_i\in \{\eta_0,\eta_1, ...,\eta_{p-1}\}$. We may identify $\mathrm{H}^0(\CO_X(1))^{\oplus n+1}$ with the space of
$(n+1)\times (n+1)$ matrices via
\[
\begin{bmatrix}
b_{st}
\end{bmatrix}_{0\le s,t \le n}
\begin{bmatrix}
z_0\\
z_1\\
\vdots\\
z_n
\end{bmatrix}\in
\mathrm{H}^0(\CO_X(1))^{\oplus n+1}.
\]
For each $j$, let
\[
V_j = \left\{
\begin{bmatrix}
b_{st}
\end{bmatrix}:
b_{st} = 0 \text{ for } \lambda_s\ne \eta_j \text{ or }
\lambda_t \ne \eta_j
\right\}
\subset \mathrm{H}^0(\CO_X(1))^{\oplus n+1}
\]
Then by \eqref{ECIAUTO100}, we see that
\[
J_X(V_0\oplus V_1\oplus\ldots\oplus V_{p-1}) \subset E_{X, d_1,\xi_1} \oplus E_{X, d_2,\xi_2} \oplus \ldots \oplus E_{X, d_c,\xi_c}.
\]
Combining with the fact that $\dim V_j = a_j^2$, we arrive at \eqref{ECIAUTO030}.
\end{proof}

\begin{lem}\label{LEMCIAUTO101}
Let $\sigma$, $E_{d,\xi}$ and $a_j$ be given as above.
Then
\begin{equation}\label{ECIAUTO031}
\binom{n+d}{d} - \dim E_{d,\xi}
> (n+1)^2 - \sum_{j=0}^{p-1} a_j^2
\end{equation}
if $d, n\ge 3$ and
\begin{equation}\label{ECIAUTO102}
\binom{n+2}{2} - \dim E_{2,\xi} \ge \frac{(n+1)^2}{2} - \frac{1}{2} \sum_{j=0}^{p-1} a_j^2
\end{equation}
for all $\xi$.
\end{lem}

\begin{proof}[Proof of Lemma \ref{LEMCIAUTO101}]
Without loss of generality, we may assume that $a_0 \ge a_1 \ge\ldots\ge a_{p-1}$.
We let $V_j = E_{1,\eta_j}$ and write
\begin{equation}\label{ECIAUTO104}
\begin{aligned}
\mathrm{H}^0(\CO(d)) &= \Sym^d V_0\oplus \sum_{j\ne 0} \Sym^{d-1} V_0 \otimes V_j\\
&\quad\oplus \Sym^d V_1\oplus \sum_{j\ne 1} \Sym^{d-1} V_1 \otimes V_j\\
&\quad\oplus ... \oplus
\Sym^d V_{p-1}\oplus \sum_{j\ne p-1} \Sym^{d-1} V_{p-1} \otimes V_j
\\
&\quad \oplus \sum_{i < j < k} \Sym^{d-2} V_{i}\otimes V_j \otimes V_k\oplus ...
\end{aligned}
\end{equation}
when $d\ge 3$.
Note that $\Sym^d V_0$, $\Sym^{d-1} V_0 \otimes V_1$, ...,
$\Sym^{d-1} V_0 \otimes V_{p-1}$ have different weights under the action of $\sigma$. Therefore,
\begin{equation}\label{ECIAUTO103}
\begin{aligned}
&\quad \dim E_{d,\xi} \cap \left(
\Sym^d V_0\oplus \sum_{j\ne 0} \Sym^{d-1} V_0 \otimes V_j
\right)
\\
&\le \max\left(
\dim \Sym^d V_0, \dim \Sym^{d-1} V_0 \otimes V_1,\right.\\
&\quad\quad\quad\quad\ldots,
\left . \dim \Sym^{d-1} V_0 \otimes V_{p-1}
\right)\\
&= \max\left(\binom{a_{0}+d-1}{d}, \binom{a_{0}+d-2}{d-1}
\binom{a_1}1\right).
\end{aligned}
\end{equation}
In other words,
\begin{equation}\label{ECIAUTO114}
\begin{aligned}
&\quad \dim \left(\Sym^d V_0\oplus \sum_{j\ne 0} \Sym^{d-1} V_0 \otimes V_j\right)
\\
&- \dim E_{d,\xi} \cap \left(
\Sym^d V_0\oplus \sum_{j\ne 0} \Sym^{d-1} V_0 \otimes V_j
\right)\\
&\ge \binom{a_{0}+d-1}{d} + \sum_{j\ne 0} \binom{a_{0}+d-2}{d-1}
\binom{a_j}1\\
&\quad - \max\left(\binom{a_{0}+d-1}{d}, \binom{a_{0}+d-2}{d-1}
\binom{a_1}1\right).
\end{aligned}
\end{equation}
By the same argument, we have
\begin{equation}\label{ECIAUTO110}
\begin{aligned}
&\quad \dim \left(\Sym^d V_i\oplus \sum_{j\ne i} \Sym^{d-1} V_i \otimes V_j\right)
\\
&- \dim E_{d,\xi} \cap \left(
\Sym^d V_i\oplus \sum_{j\ne i} \Sym^{d-1} V_i \otimes V_j
\right)\\
&\ge \binom{a_{i}+d-1}{d} + \sum_{\substack{j\ne i\\ j\ge 1}} \binom{a_{i}+d-2}{d-1}
\binom{a_j}1
\end{aligned}
\end{equation}
for all $i\ge 1$ and
\begin{equation}\label{ECIAUTO111}
\begin{aligned}
&\quad \dim \left(\Sym^{d-2} V_i\otimes V_j \otimes \sum_{k>j} V_k\right)
\\
&- \dim E_{d,\xi} \cap \left(
\Sym^{d-2} V_i\otimes V_j \otimes \sum_{k>j} V_k
\right)\\
&\ge \sum_{k > j+1} \binom{a_{i}+d-3}{d-2}
\binom{a_j}1 \binom{a_k}{1}
\end{aligned}
\end{equation}
for all $i < j$. Combining \eqref{ECIAUTO114}, \eqref{ECIAUTO110}
and \eqref{ECIAUTO111}, we conclude
\begin{equation}\label{ECIAUTO035}
\begin{aligned}
\binom{n+d}{d} - \dim E_{d,\xi}
&\ge
\sum_i
\binom{a_i + d - 1}{d}
+ \sum_{\genfrac{}{}{0pt}{}{i\ne j}{j\ge 1}} \binom{a_{i}+d-2}{d-1} a_j\\
&\quad
- \max\left(\binom{a_{0}+d-1}{d}, \binom{a_{0}+d-2}{d-1}a_1\right)\\
&\quad + \sum_{i < j < k-1} \binom{a_i+d-3}{d-2} a_j a_k.
\end{aligned}
\end{equation}
Note that the right hand side of~\eqref{ECIAUTO031} is simply $\sum_{i< j} 2a_i a_j$. So it suffices to verify the following:
\begin{equation}\label{ECIAUTO112}
\begin{aligned}
&\binom{a_{i}+d-2}{d-1} a_j + \binom{a_{j}+d-2}{d-1} a_i \ge 2a_ia_j
\text{ for } 1\le i < j\\
&\binom{a_{0}+d-2}{d-1} a_j + \binom{a_{j}+d-1}{d} \ge 2 a_0 a_j
\text{ for } j\ge 2\\
&\min\left(\binom{a_{0}+d-1}{d}, \binom{a_{0}+d-2}{d-1}a_1\right)
+ \binom{a_{1}+d-1}{d}
 \ge 2a_0a_1\\
&\sum_{i < j < k-1} \binom{a_i+d-3}{d-2} a_j a_k \ge 0.
\end{aligned}
\end{equation}
Therefore, we conclude
\begin{equation}\label{ECIAUTO113}
\binom{n+d}{d} - \dim E_{d,\xi}
\ge \sum_{i<j} 2a_ia_j = (n+1)^2 - \sum_{j=0}^{p-1} a_j^2.
\end{equation}
And the equality in \eqref{ECIAUTO113} holds only if all equalities hold in \eqref{ECIAUTO114}, \eqref{ECIAUTO110}, \eqref{ECIAUTO111}
and \eqref{ECIAUTO112}, which cannot happen. So~\eqref{ECIAUTO031} follows.

The proof of \eqref{ECIAUTO102} is quite straightforward and we leave it to the readers.
\end{proof}

\begin{lem}\label{LEMCIAUTO102}
Let $\sigma$, $E_{d,\xi}$, $a_j$ and $\mu$ be given as above.
Then
\begin{equation}\label{ECIAUTO120}
\begin{aligned}
\dim E_{d,\xi} &\le \sum_{k=0}^{n+1 - \mu} \frac{(\mu-1)^k}{\mu^{k+1}} \binom{n+d-k}{d}
\\
&\quad\quad +  \frac{(\mu-1)^{n+1-\mu}}{\mu^{n+2-\mu}} \binom{d+\mu-1}{d+1}
\end{aligned}
\end{equation}
for all $\xi$.
\end{lem}

\begin{proof}[Proof of Lemma \ref{LEMCIAUTO102}]
We let $e(d,a_0,a_1,...,a_{p-1})$ be the maximum of $\dim E_{d,\xi}$ for all
$\xi$ and all $\sigma$ such that $a_j = \dim E_{1,\eta_j}$.

Let us choose homogeneous coordinates $(z_0, z_1,\ldots,z_n)$ of $\mathbb{P}^n$ such that $\sigma(z_i) = \lambda_i z_0$ and $\lambda_0, \lambda_1, ..., \lambda_{\mu-1}$ are distinct.
We observe that
for each monomial
$z_0^{b_0} z_1^{b_1} ... z_n^{b_n}$ with $b_0 > 0$,
$z_0^{b_0} z_1^{b_1} ... z_n^{b_n} (z_i/z_0)$ have different weights under the action of $\sigma
$ for
$i = 0,1,...,\mu - 1$. This observation leads to
the recursive inequality
\begin{equation}\label{ECIAUTO121}
e(d,a_0,a_1,\ldots,a_{p-1}) \le
\frac{1}\mu \binom{n+d}{d} + \frac{\mu-1}{\mu} e(d, a_0 -1, a_1, \ldots,a_{p-1})
\end{equation}
where we assume that $a_0 = \max(a_j)$ and $\lambda_0 = \eta_0$.

It is easy to check that \eqref{ECIAUTO120} follows from \eqref{ECIAUTO121} by induction.
\end{proof}

Now we are ready to prove Theorem \ref{THMCIAUTO000}.
By induction, we just have to deal with the following cases:
\begin{itemize}
\item $d_1 = d_2 = \cdots = d_c = d$ ($d\ge 3$ or $d, c\ge 2$)
\item $2 = d_1 = d_2 < d_3 = \cdots = d_c = d$ and
\item $2=d_1 < d_2= \cdots = d_c = d$.
\end{itemize}
We only need to prove \eqref{ECIAUTO040} in these cases.

\subsection*{The case $d_1 = d_2 = \cdots = d_c = d$}

By~\eqref{ECIAUTO030}, we have
\begin{equation}\label{ECIAUTO029}
\begin{aligned}
&\quad \dim\left(\IM J_X + \left(E_{X,d,\xi_1} \oplus E_{X,d,\xi_2} \oplus \ldots \oplus E_{X,d,\xi_c}
\right)\right)\\
&\le (n+1){}^2 - \sum_{j=0}^{p-1} a_j^2 + \sum_{i=1}^c \dim E_{X,d,\xi_i}.
\end{aligned}
\end{equation}
Thus, in order to prove~\eqref{ECIAUTO040}, it suffices to prove
\begin{equation}\label{ECIAUTO115}
(n+1){}^2 - \sum_{j=0}^{p-1} a_j^2 + \sum_{i=1}^c \dim E_{X,d,\xi_i}
< c h^0(\CO_X(d)) = c \binom{n+d}{d} - c^2.
\end{equation}
If $\xi_{i_1}, \xi_{i_2}, \ldots, \xi_{i_l}$ are distinct, we observe that
\begin{equation}\label{ECIAUTO105}
\dim E_{X,d,\xi_{i_1}} + \dim E_{X,d,\xi_{i_2}} + \ldots +
\dim E_{X,d,\xi_{i_l}} \le h^0(\CO_X(d)).
\end{equation}
We write $\{1,2,...,c\} = B_1\sqcup B_2\sqcup \ldots$
as a disjoint union of
sets $B_i$ such that $\xi_l = \xi_m$ if and only if $\{l,m\}\subset B_i$
for some $i$. Let $b_i = |B_i|$.
Note that
\begin{equation}\label{ECIAUTO106}
\dim E_{d,\xi_l} - \dim E_{X,d,\xi_l} \ge b_i \text{ for } l\in B_i.
\end{equation}
Suppose that $b_1\ge b_2\ge \ldots$.
Combining~\eqref{ECIAUTO105} and~\eqref{ECIAUTO106}, we derive
\[
\begin{aligned}
\sum_{i=1}^c \dim E_{X,d,\xi_i} &\le
b_2 h^0(\CO_X(d)) + \sum_{l\in B} \dim E_{X,d,\xi_l}
\\
&\le b_2 h^0(\CO_X(d)) + \sum_{l\in B} \dim E_{d,\xi_l} - b_1(b_1-b_2)
\end{aligned}
\]
for all $B\subset B_1$ with $|B| = b_1 - b_2$. Note that $b_1 + b_2\le c$.
So it comes down to verifying
\begin{equation}\label{ECIAUTO107}
(n+1){}^2 - \sum_{j=0}^{p-1} a_j^2 + (b_1 - b_2) \dim E_{d,\xi}
+ 2b_1b_2
< b_1 \binom{n+d}{d}
\end{equation}
for all $\xi$, $b_1\ge b_2\in \mathbb{N}$ and $b_1+b_2 = c$. Here we use $\mathbb{N}$ for the set of nonnegative integers.

It is easy to check that \eqref{ECIAUTO107} follows from \eqref{ECIAUTO031}
when $d\ge 3$ for all $n\ge c+2\ge 3$ and from \eqref{ECIAUTO102}
when $d=2$ for all $n\ge c+1\ge 4$.
This proves \eqref{ECIAUTO107} and hence \eqref{ECIAUTO115}.

We are left with the only exceptional case $d=c=2$, which is settled by the following proposition.

\begin{prop}\label{PROPCIAUTO2QUADRICS}
For a general smooth intersection $X$ of two quadrics in $\PP_k^n$,
$\mathrm{Aut}(X) = (\BZ/2\BZ){}^{n}$ if $n\ge 4$ and $\chari(k)\ne 2$.
\end{prop}

\begin{proof}
Let $W \subset \PP^n\times \PP^1$ be a pencil of quadrics whose base locus
is $X$. Since $\mathrm{Aut}(X)$ acts $\mathrm{H}^0(I_X(2))$, every $\sigma\in \mathrm{Aut}(X)$ induces
an automorphism of $W$ with diagram
\begin{equation}\label{ECIAUTO038}
\xymatrix{
W \ar[r]^{\sigma} \ar[d] & W \ar[d]\\
\PP^1 \ar[r]^{g} & \PP^1
}
\end{equation}
Let $D = \{ b\in \PP^1: W_b \text{ singular}\}$ be the discriminant locus
of the pencil $W$. Clearly, $g(D) = D$.

We can make everything explicit. Let $Q_A$ and $Q_B$ be two members of the pencil
whose defining equations are given by two symmetric matrices $A$ and $B$.
Obviously, we may take $A = I$ and assume that a member of pencil is given by
$B - tI$. Then $D$ consists of exactly the eigenvalues $\lambda_0,\lambda_1,\ldots,
\lambda_n$ of $B$. By~\cite[Proposition 2.1]{R}, the set $D = \{\lambda_0,\lambda_1,\ldots,
\lambda_n\}$ is a set of $n+1\ge 5$ general points on $\PP^1$. Therefore,
$g\in \mathrm{Aut}(\PP^1)$ sending $D$ to $D$ must be the identity map
and $\sigma$ preserves the fiberation $W/\PP^1$, i.e.,
it induces an automorphism of $W_b$ for each $b$.
In particular, $\sigma$ maps $W_b$ to $W_b$ for each $b\in D$.
Every $W_b$ has exactly one node, i.e., an ordinary double point
$p_i$ for $b = \lambda_i$, which is given by the eigenvector of $B$
corresponding to $\lambda_i$. Clearly, $\sigma$ fixes $p_0, p_1, \ldots, p_n$.

The matrix $B$ can always be orthogonally diagonalized
(this follows from \cite[Proposition 2.1]{R}). Therefore,
$W$ is given by
\begin{equation}\label{ECIAUTO034}
\begin{split}
(\lambda_0 z_0^2 + \lambda_1 z_1^2 + \cdots + \lambda_n z_n^2)
- t(z_0^2 + z_1^2 + \cdots + z_n^2) = 0
\end{split}
\end{equation}
after some action of $\mathrm{Aut}(\PP^n)$. Since $\sigma$ fixes $p_0, p_1, \ldots, p_n$,
it has to be a $(k^*){}^{n+1}$ action on $(z_0, z_1, \ldots, z_n)$. And since
it preserves the fibers of $W/\PP^1$, it must be
\begin{equation}\label{ECIAUTO039}
\sigma(z_0, z_1, \ldots, z_n) = (\pm z_0, \pm z_1, \ldots, \pm z_n)
\end{equation}
and we then conclude that $\mathrm{Aut}(X) = (\BZ/2\BZ){}^{n+1}/\{\pm 1\}=(\BZ/2\BZ){}^{n}$.
\end{proof}

%%%\begin{example}[A (2,2)-complete intersection with an extra automorphism]
%%%Let $k$ be an algebraically closed field. Assume that the characteristic of $k$
%%%is not equal to 2. Let $n > 0$ be an integer prime to the characteristic of $k$.
%%%Let $\eta$ be a primitive $n$th root of unit. Consider the quadric hypersurfaces
%%%\begin{equation*}
%%%\textstyle
%%%D = \{z \in \mathbb{P}^{n-1}: \sum_{i=0}^{n-1} z_i^2 = 0\} \text{ and }
%%%E = \{z \in \mathbb{P}^{n-1}: \sum_{i=1}^{n-1} \eta^{i}z_i^2 = 0\}
%%%\end{equation*}
%%%and the linear transformation $\sigma$ on $\mathbb{P}^{n-1}$ given by
%%%\begin{equation*}
%%%z_i \mapsto \eta^{1/2} z_{i+1},
%%%\end{equation*}
%%%where the subscripts are understood as integers modulo $n$. Then $\sigma$ sends
%%%\begin{equation*}
%%%\sum_{i \in \mathbb{Z}/n} z^2_i \mapsto \sum_{i \in \mathbb{Z}/n} \eta\cdot z_{i+1}^2 \text{ and }
%%%\sum_{i \in \mathbb{Z}/n} \eta^{i}z_{i}^2 \mapsto \sum_{i \in \mathbb{Z}/n} \eta^{i}\cdot \eta\cdot z_{i+1}^2.
%%%\end{equation*}
%%%It follows that $\sigma$ is simultaneously an automorphism for both $D$ and $E$.
%%%Hence it induces an automorphism of the (2,2)-complete intersection $X = D \cap
%%%E$.
%%%It is elementary to check that $X$ is nonsingular.
%%%This automorphism is not induced from the standard automorphism of $X$ obtained
%%%by ``flipping the factors''.
%%%\end{example}

\subsection*{The case $2 = d_1 = d_2 < d_3 = \cdots = d_c = d$}

By~\eqref{ECIAUTO030}, it suffices to prove
\begin{equation}\label{ECIAUTO041}
\begin{split}
&\quad (n+1){}^2 - \sum_{j=0}^{p-1} a_j^2 +
\dim E_{X,2,\xi_1} + \dim E_{X,2,\xi_2}
+ \sum_{i=3}^c \dim E_{X, d,\xi_i}\\
&< 2h^0(\CO_X(2)) + (c-2) h^0(\CO_X(d))
\end{split}
\end{equation}
where
\[
h^0(\CO_X(d)) =
\binom{n+d}{n} - 2\binom{n+d-2}{n} + \binom{n+d-4}{n} - (c-2).
\]
Applying the same argument as before to $E_{X,d,\xi_i}$ for $i\ge 3$, we can further reduce \eqref{ECIAUTO041} to
\begin{equation}\label{ECIAUTO109}
\begin{aligned}
&\quad (n+1){}^2 - \sum_{j=0}^{p-1} a_j^2 +
\sum_{i=1}^2 \dim E_{X,2,\xi_i} + (b_1 - b_2) \dim E_{d,\xi}  + 2b_1b_2\\
&< 2h^0(\CO_X(2)) + b_1 \left(
\binom{n+d}{n} - 2\binom{n+d-2}{n} + \binom{n+d-4}{n}\right)
\end{aligned}
\end{equation}
for all $\xi$, $b_1\ge b_2\in \mathbb{N}$ and $b_1 + b_2 = c-2$.

For $E_{X,2,\xi_1}$ and $E_{X,2,\xi_2}$,
it follows from \eqref{ECIAUTO102} that
\begin{equation}\label{ECIAUTO124}
\begin{aligned}
&\quad (n+1){}^2 - \sum_{j=0}^{p-1} a_j^2 +
\sum_{i=1}^2 \dim E_{X,2,\xi_i}
\\
&\le (n+1){}^2 - \sum_{j=0}^{p-1} a_j^2 + \sum_{i=1}^2 \dim E_{2,\xi_i} - 2\\
& \le 2\binom{n+2}{2} - 2 = 2h^0(\CO_X(2)) + 2.
\end{aligned}
\end{equation}

Suppose that $\mu \ge 3$. We have
\begin{equation}\label{ECIAUTO125}
\binom{n+d}{d} - \dim E_{d, \xi} \ge \frac{2}3 \binom{n+d-1}{d-1}
+ \frac{4}9 \binom{n+d-2}{d-1}
\end{equation}
for $n\ge 3$ by \eqref{ECIAUTO120}. It is easy to check that
\begin{equation}\label{ECIAUTO126}
\frac{2}3 \binom{n+d-1}{d-1}
+ \frac{4}9 \binom{n+d-2}{d-1} > 2\binom{n+d-2}{n} + 2
\end{equation}
for all $d\ge 3$ and $n\ge 4$. Thus, \eqref{ECIAUTO109}
follows from \eqref{ECIAUTO124}, \eqref{ECIAUTO125}
and \eqref{ECIAUTO126} for all $d\ge 3$ and $n\ge c+1 \ge 4$.

Suppose that $\mu = 2$. If $\xi_1 \ne \xi_2$,
\begin{equation}\label{ECIAUTO042}
\dim E_{X,2,\xi_1} + \dim E_{X,2,\xi_2} \le h^0(\CO_X(2))
\end{equation}
and hence
\begin{equation}\label{ECIAUTO108}
\begin{aligned}
&\quad (n+1){}^2 - \sum_{j=0}^{p-1} a_j^2 +
\sum_{i=1}^2 \dim E_{X,2,\xi_i}\\
&\le (n+1){}^2 - \sum_{j=0}^{p-1} a_j^2 +
h^0(\CO_X(2))\\
&\le \frac{(n+1)^2}{2} + h^0(\CO_X(2)) = 2h^0(\CO_X(2))
\end{aligned}
\end{equation}
for $n\ge 3$. On the other hand, if $\xi_1 = \xi_2$,
\begin{equation}\label{ECIAUTO032}
\dim E_{X,2,\xi_1} + \dim E_{X,2,\xi_2} \le \dim E_{2,\xi_1} + \dim E_{2,\xi_2} - 4
\end{equation}
and \eqref{ECIAUTO108} still holds by \eqref{ECIAUTO102}.

By \eqref{ECIAUTO120}, we have
\begin{equation}\label{ECIAUTO122}
\begin{aligned}
\binom{n+d}{d} - \dim E_{d, \xi} &\ge \frac{1}2 \binom{n+d-1}{d-1}
+ \frac{1}4 \binom{n+d-2}{d-1}\\
&\quad + \frac{1}8 \binom{n+d-3}{d-1}
\end{aligned}
\end{equation}
for $\mu \ge 2$ and $n\ge 3$. It is easy to check that
\begin{equation}\label{ECIAUTO116}
\frac{1}2 \binom{n+d-1}{d-1}
+ \frac{1}4 \binom{n+d-2}{d-1}
+ \frac{1}8 \binom{n+d-3}{d-1} > 2\binom{n+d-2}{n}
\end{equation}
for all $d\ge 3$ and $n\ge 4$. Thus, \eqref{ECIAUTO109}
follows from \eqref{ECIAUTO108}, \eqref{ECIAUTO122} and \eqref{ECIAUTO116}
for all $d\ge 3$ and $n\ge c+1 \ge 4$.

\subsection*{The case $2=d_1 < d_2 = d_3 = \cdots= d_c = d$}

By~\eqref{ECIAUTO030}, it suffices to prove
\begin{equation}\label{ECIAUTO037}
\begin{aligned}
&\quad (n+1){}^2 - \sum_{j=0}^{p-1} a_j^2 +
\dim E_{X,2,\xi_1} + \sum_{i=2}^c \dim E_{X,d,\xi_i}
\\
&<
h^0(\CO_X(2)) + (c-1) h^0(\CO_X(d))
\end{aligned}
\end{equation}
where
\[
h^0(\CO_X(d))
= \binom{n+d}{n} - \binom{n+d-2}{n} -(c-1).
\]
By \eqref{ECIAUTO102}, we have
\begin{equation}\label{ECIAUTO118}
\begin{aligned}
&\quad (n+1){}^2 - \sum_{j=0}^{p-1} a_j^2 + \dim E_{X,2,\xi_1} - h^0(\CO_X(2))
\\
&\le (n+1){}^2 - \sum_{j=0}^{p-1} a_j^2 + \dim E_{2,\xi_1} - 1 - \left(\binom{n+2}2 -1\right)
\\
&\le \frac{(n+1)^2}{2} - \frac{1}2\sum_{j=0}^{p-1} a_j^2.
\end{aligned}
\end{equation}
So \eqref{ECIAUTO037} holds if
\begin{equation}\label{ECIAUTO117}
\frac{(n+1)^2}{2} - \frac{1}2\sum_{j=0}^{p-1} a_j^2 + \sum_{i=2}^c \dim E_{X,d,\xi_i}
< (c-1) h^0(\CO_X(d)).
\end{equation}
Applying the same argument as before to $E_{X,d,\xi_i}$ for $i\ge 2$, we can further reduce \eqref{ECIAUTO117} to
\begin{equation}\label{ECIAUTO036}
\begin{split}
&\quad\frac{(n+1)^2}{2} - \frac{1}2\sum_{j=0}^{p-1} a_j^2
+ (b_1 -b_2) \dim E_{d,\xi} + 2b_1b_2\\
&< b_1 \left(\binom{n+d}{n} - \binom{n+d-2}{n}\right)
\end{split}
\end{equation}
for all $\xi$, $b_1\ge b_2\in \mathbb{N}$ and $b_1 + b_2 = c-1$.

By \eqref{ECIAUTO122} and \eqref{ECIAUTO116}, we have
\begin{equation}\label{ECIAUTO123}
\begin{aligned}
&\quad \left(b_1 -b_2- \frac{1}2\right) \dim E_{d,\xi} + 2b_1b_2
\\
&< \left(b_1 - \frac{1}2\right) \binom{n+d}{n} - b_1 \binom{n+d-2}{n}
\end{aligned}
\end{equation}
for all $d\ge 3$ and $n\ge c+2\ge 4$.
Combining \eqref{ECIAUTO123} with \eqref{ECIAUTO031}, we obtain \eqref{ECIAUTO036}.
This finishes the proof of Theorem~\ref{THMCIAUTO000}.

\section{Equivariant Deformations}
\label{sec:equiv-deform}

In this section, we shall only consider complex algebraic varieties. We shall
use equivariant deformation theory to show that the automorphism group of a
complete intersection (except some cases) over \(\mathbb{C}\) acts on the
singular cohomology groups in a faithful fashion.

We begin by recalling some basics about equivariant deformation theory.
Ringed spaces with a group action is a particular case of a ringed topos.
Therefore the so-called equivariant deformation theory is just a specialization
of the theory of cotangent complex for ringed topos as developed by
Illusie~\cite{Ill}.

\refstepcounter{thm}
\subsection*{\thetheorem}
Suppose \(Y\) is a scheme over \(\mathbb{C}\) with an action
\begin{equation*}
G\times Y\rightarrow Y
\end{equation*}
by an abstract group \(G\). Then the \emph{equivariant cotangent complex}
\(\mathbb{L}_{Y}^{G}\) is a bounded-below complex of quasi-coherent sheaves on
\(Y\) with an action of \(G\) covering the action of \(G\) on \(Y\). The
underlying complex of \(\mathbb{L}_{Y}^{G}\) is the usual cotangent complex
\(\mathbb{L}_{Y}\). The \(G\) structure induces a Hochschild--Serre spectral
sequence with
\begin{equation*}
E_{2}^{i,j} = \mathrm{H}^{i}(G,\mathrm{Ext}_{\mathcal{O}_{Y}}^{j}(\mathbb{L}_{Y},\mathcal{O}_{Y}))
\end{equation*}
abutting to
\begin{equation*}
\mathrm{Ext}^{i+j}_{G\text{-}\mathcal{O}_{Y}}(\mathbb{L}_{Y}^{G},\mathcal{O}_{Y}),
\end{equation*}
here we regard \(\mathcal{O}_Y\) as an \(\mathcal{O}_Y\){}-module with a trivial
action of \(G\).

\refstepcounter{thm}
\subsection*{\thetheorem}\label{sec:equivariant-obs-class}
When \(G\) is \textit{finite}, the Hochschild--Serre spectral sequence
degenerates, and the edge map of the spectral sequence
\begin{equation}
\label{eq:identifying-obs-spaces}
\text{Ext}^i_{G\text{-}\mathcal{O}_Y}(\mathbb{L}_{Y}^G, \mathcal{O}_Y)
\to \mathrm{H}^{0}(G,\text{Ext}^i_{\mathcal{O}_Y}(\mathbb{L}_{Y}, \mathcal{O}_Y))
=\mathrm{Ext}^{i}_{\mathcal{O}_{Y}}(\mathbb{L}_{Y},\mathcal{O}_{Y})^{G}.
\end{equation}
is therefore an isomorphism for all \(i \geq 0\).

The case when \(i=1,2\) are mostly important for us.
Recall that the space of infinitesimal deformations of a \(G\){}-variety \(Y\)
is identified with
\(\mathrm{Ext}^{1}_{G\text{-}\mathcal{O}_{Y}}(\mathbb{L}_{Y}^{G},\mathcal{O}_{Y})\).
When \(G\) is the trivial group, and \(Y\) is smooth over \(\mathbb{C}\), the
space of infinitesimal deformations of \(Y\) is thus
identified with the cohomology group \(\mathrm{H}^{1}(Y,T_{Y})\) of the tangent
bundle. When \(G\) is finite, and \(Y\) is smooth, the above degeneracy tells us
that the space of infinitesimal deformations of \(Y\) with
an action of \(G\) is identified with linear subspace
\(\mathrm{H}^{1}(Y,T_{T})^{G}\) of \(\mathrm{Def}(Y)=\mathrm{H}^{1}(Y,T_{Y})\).

The case \(i=2\) is related to the obstruction theory. Suppose that we have a
small extension of Artinian \(\mathbb{C}\){}-algebras
\begin{equation*}
0 \to \mathbb{C} \rightarrow R \rightarrow R' \rightarrow 0,
\end{equation*}
and a flat \(R'\){}-scheme \(Y'\) with an action of \(G\), such that
\(Y' \otimes_{R'} \mathbb{C} = Y\), as schemes with \(G\){}-actions. Just like
the usual deformation theory, attached to the above deformation situation, there
is an element
\(\mathrm{obs}(Y',G)\in\mathrm{Ext}^{2}_{G\text{-}\mathcal{O}_{Y}}(\mathbb{L}_{Y}^{G},\mathcal{O}_{Y})\),
known as the \textit{equivariant obstruction class}, such that the deformation
\(Y'\) extends to a deformation over \(R\) if and only if
\(\mathrm{obs}(Y',G)=0\).

Under the identification~\eqref{eq:identifying-obs-spaces}, the equivariant
obstruction class is sent to the usual obstruction class. Therefore if the
infinitesimal deformations of \(Y\) are all unobstructed (i.e., the obstruction
class vanishes for any deformation situation), the \(G\){-}equivariant
infinitesimal deformations of \(Y\) are all unobstructed as well.

\refstepcounter{thm}
\subsection*{\thetheorem}
Now we assume that \(Y\) is a smooth projective variety of dimension \(n\). For
positive integers \(p,q\) such that \(p+q=n\). Then the map contraction map
\begin{equation*}
T_{Y} \otimes \Omega_{Y}^{p} \to \Omega_{Y}^{p-1}
\end{equation*}
gives rise to a map
\begin{equation}
\label{eq:inf-torelli-map}
\mathrm{H}^{1}(Y,T_{Y}) \to
\mathrm{Hom}(\mathrm{H}^q(Y,\Omega_Y^p),\mathrm{H}^{q+1}(Y,\Omega_Y^{p-1})).
\end{equation}
We say that \emph{the infinitesimal Torelli theorem holds} for \(Y\) if there
exists \(p,q\) such that the above arrow is injective.

\begin{lemma}%
\label{lemma:automorphism-fixing-cohomology}
Let \(Y\) be a smooth proper variety of dimension \(n\). Assume that
\begin{enumerate}
\item the infinitesimal Torelli theorem holds for \(Y\), and
\item \(\mathrm{H}^{n}(g)=\mathrm{Id}\) for all \(g \in G\).
\end{enumerate}
Then \(\mathrm{H}^{1}(Y,T_Y)=\mathrm{H}^{1}(Y,T_{Y})^{G}\).
\end{lemma}

\begin{proof}
Choose \(p\) and \(q\) so that the map~\eqref{eq:inf-torelli-map} is injective.
Consider the following commutative diagram
\begin{equation*}
\begin{tikzcd}
\mathrm{H}^1(Y, T_Y)^G \ar[hook]{d}\ar{r} & \mathrm{Hom}(\mathrm{H}^q(Y,\Omega_Y^p),
\mathrm{H}^{q+1}(Y,\Omega_Y^{p-1}))^G
\ar[hook]{d}\\
\mathrm{H}^1(Y, T_Y) \ar[hook]{r} & \mathrm{Hom}(\mathrm{H}^q(Y,\Omega_Y^p),\mathrm{H}^{q+1}(Y,\Omega_Y^{p-1}))
\end{tikzcd}
\end{equation*}
Since \(G\) acts trivially on the cohomology group
\(\mathrm{H}^{n}(Y,\mathbb{C})\), it induces a trivial action on the Hodge
groups. Therefore the right vertical arrow in the above diagram is an equality.
It follows that the image of any element in \(\mathrm{H}^{1}(Y,T_Y)\) is
\(G\){}-invariant, and therefore the infinitesimal Torelli property implies that
it falls in \(\mathrm{H}^{1}(Y,T_{Y})^G\).
\end{proof}

\refstepcounter{thm}
\subsection*{\thetheorem.~Smooth Fano varieties and smooth Calabi--Yau varieties}
\label{subsec:cy-fano}
Assume that \(Y\) is a smooth Fano variety. Then the cotangent complex
\(\mathbb{L}_{Y}\) is equal to the sheaf \(\Omega_{Y}^{1}\). The Kodaira
vanishing theorem then implies that
\begin{equation*}
\mathrm{Ext}^2(\mathbb{L}_{Y},\mathcal{O}_{Y})=\mathrm{H}^2(Y,T_Y)=
\mathrm{H}^{\dim Y-2}(Y,\Omega_Y^1\otimes \omega_Y)=0.
\end{equation*}
Therefore the deformation of \(Y\) is unobstructed. When \(Y\) is a smooth
Calabi--Yau variety, it is well-known theorem (see e.g.,~\cite{Kawa} for an
algebraic proof based on the work of Ziv Ran~\cite{ran:t1-lifting}) that the
deformation of \(Y\) is unobstruced as well. In both cases above, according
to~\eqref{sec:equivariant-obs-class}, if \(Y\) also acquires an action of a
finite group \(G\), the equivariant deformation of \(Y\) is unobstructed as
well.

\begin{proposition}
\label{ci-faithful-cy-fano}
Let \(Y\) be a smooth complete intersection in \(\mathbb{P}_{\mathbb{C}}^r\) of
type \((d_1,\ldots,d_c)\), \(d_i \geq 2\). Let \(n=r-c\). Assume that
\begin{itemize}
\item \(Y\) is not an elliptic curve,
\item \(Y\) is Fano or Calabi--Yau,
\item the infinitesimal Torelli theorem holds for \(Y\),
\end{itemize}
Then the map
\begin{equation*}
\mathrm{Aut}(Y) \to \mathrm{H}^{n}(X,\mathbb{Q})
\end{equation*}
is injective.
\end{proposition}

\begin{proof}
The proposition is well-known for K3 surfaces. If \(Y\) is a surface with
nontrivial canonical bundle, or the
dimension of \(Y\) is at least \(3\), then the group of automorphisms of \(Y\)
is the same as the group of linear automorphisms with respect to the projective
embedding, thus it suffices to prove the group \(\mathrm{Aut}_{L}(Y)\) of linear
automorphisms of \(Y\) acts faithfully on the \(\mathrm{H}^{n}(Y,\mathbb{Q})\).
Moreover, the condition that \(Y\) satisfies the infinitesimal Torelli
theorem implies that \(Y\) can not be a quadric or an intersection of two quadrics.
Therefore, Theorem~\ref{THMCIAUTO000} can be applied to the present situation.

Let \(G\) be the kernel of
\begin{equation*}
\mathrm{Aut}(Y) \to \mathrm{GL}(\mathrm{H}^{n}(Y,\mathbb{Q})).
\end{equation*}
We want to show that \(G\) is trivial.
Since \(Y\) is a smooth Fano variety or a smooth Calabi--Yau variety satisfying
the infinitesimal Torelli theorem, we know
from~\eqref{lemma:automorphism-fixing-cohomology} and~\eqref{subsec:cy-fano} that
the deformation space of \(Y\) is smooth and equals the deformation space of
\(Y\) with \(G\){}-action. Therefore the automorphisms \(g\) in \(G\) is
extendable to any nearby deformation \(Y'\) of \(Y\). It follows from
Theorem~\ref{THMCIAUTO000} that \(g\) is a specialization of the identity map.
Therefore, \(g\) is the identity.
\end{proof}

\begin{proposition}
\label{ci-faithful-gt}
Let \(Y\) be a smooth complete intersection in \(\mathbb{P}_{\mathbb{C}}^r\) of
dimension \(n\). Assume that \(Y\) is of general type. Then the map
\begin{equation*}
\mathrm{Aut}(Y) \to \mathrm{H}^{n}(X,\mathbb{Q})
\end{equation*}
is injective.
\end{proposition}

\begin{proof}
Let \(g \in G\) is an element acting trivially on the cohomology of \(Y\). The
induced action of $g$ on the ring
\[
S_Y = \bigoplus_{n\geq 0}\mathrm{Sym}^d(\mathrm{H}^0(Y,\omega_Y))
\]
is trivial. Since $Y$ is a complete intersection of general type, $Y$ has very
ample canonical bundle hence $Y\subseteq \mathrm{Proj} (S_Y)$. Since \(g\) acts
trivially on \(\mathrm{Proj}(S_Y)\), $g$ induces the identity on $Y$.
\end{proof}

The case of a cubic surface can be dealt separately, and we can directly prove a
version that holds true for any algebraically closed field.

\begin{lemma}%
\label{lemma:cubic-auto}
Let \(Y\) be a smooth cubic surface over an algebraically closed field. Let
\(\mathrm{H}^{2}(Y)\) be
\begin{itemize}
\item either the singular cohomology \(\mathrm{H}^{2}(Y,\mathbb{Q})\) if the
  ambient field is \(\mathbb{C}\), or
\item the \(\ell\){}-adic cohomology \(\mathrm{H}_{\textup{ét}}^{2}(Y,\mathbb{Q}_{\ell})\)
  if the ambient field has characteristic different from \(\ell\), or
\item the crystalline cohomology \(\mathrm{H}_{\textup{cris}}^2(Y/W(k))[1/p]\)
  when \(k\) has characteristic \(p>0\).
\end{itemize}
If \(g\) is an automorphism of \(Y\) such that \(g^{\ast}=\mathrm{Id}\) on
\(\mathrm{H}^{2}(Y)\). Then \(g=\mathrm{Id}\).
\end{lemma}

\begin{proof}
It is a classical theorem
(see~e.g.,~\cite[Chapter~V,~Corollary~4.7]{hartshorne:ag}) that \(Y\) is the
projective plane blown up six point \(p_1, \ldots,p_6\) such that the no three
points among the six are colinear and the six points are not contained in a
single conic. Let \(E_1,\ldots,E_6\) be the exceptional divisors on \(Y\)
corresponding to the six points. Let \(g\) be an isomorphism of \(Y\) fixing the
cohomology of \(Y\). Then \(g^{\ast}[E_i]=[E_i]\). For divisors on \(Y\), being
rationally equivalent is the same as being homologically equivalent. Therefore
the divisor \(g^{\ast}E_i\) is necessarily equal to \(E_i\). This implies that
the automorphism \(g\) commutes with the blowing up map, and thus descends to an
automorphism of \(\mathbb{P}^{2}\) fixing the six points \(p_1, \ldots, p_{6}\).
Since any three points among the six points are not colinear, there is a
projective transformation \(g'\) sending \(p_1\) to \([1,0,0]\), sending \(p_2\)
to \([0,1,0]\), sending \(p_3\) to \([0,0,1]\) and \(p_4\) to \([1,1,1]\). On
the other hand, the only projective transformation that fixes these four points
is the identity. Since the descendent \(g'\) of \(g\) fixes \(p_1,\ldots,p_4\),
it has to be the identity. It remains to check that \(g\) induces the identity
map on the exceptional divisors. If \(x\) is a point in \(E_i \subset Y\) lying
above \(p_i\). Then there exists a line \(L\) through \(p_i\) such that the
proper transform of \(L\) intersects with \(E_i\) at a unique point that is
\(x\). Since \(g'\) fixes \(L\), \(g\) fixes its proper transform by continuity.
Therefore \(g\) fixes \(x\), as desired.
\end{proof}

\refstepcounter{thm}
\subsection*{\thetheorem}
To summarize, we have shown that if \(Y\) is any smooth complete intersection of
in a projective space \(\mathbb{P}_{\mathbb{C}}^{r}\) type \((d_1,\ldots,d_c)\)
with \(d_i \geq 2\), such that
\begin{itemize}
\item \(Y\) is not an elliptic curve,
\item \(Y\) is not a quadric,
\item \(Y\) is not an intersection of two quadrics,
\end{itemize}
then the automorphism group of \(Y\) acts faithfully on the singular cohomology
group \(\mathrm{H}^n(Y,\mathbb{Q})\).

Hence \(\mathrm{Aut}(Y)\) acts faithfully on the singular cohomology with
\(\mathbb{Q}_{\ell}\)-coefficient. Since singular and étale cohomology groups
agree for complex algebraic varieties for the coefficient \(\mathbb{Q}_{\ell}\),
Theorem~\ref{autocoh}(2) holds true for \(k=\mathbb{C}\).

\section{Automorphism and Cohomology}
\label{sec:aut-and-coh}
Throughout this section we shall assume that \(k\) is an algebraically closed
field of characteristic \(p>0\). Let \(W(k)\) be the ring of Witt vectors of \(k\). Let \(\ell\) be a prime
number different from \(p\). Let \(X\) be a smooth complete intersection inside
a projective space. In this section, we prove that, besides some obvious
exceptions, the action of \(\mathrm{Aut}(X)\) on the middle crystalline
cohomology, as well as the middle \(\ell\){-}adic étale cohomology of \(X\), is
faithful.

\refstepcounter{thm}
\subsection*{\thetheorem}
Let \(X\) be a smooth projective variety over \(k\).
Let \(g \in \mathrm{Aut}(X)\) be an automorphism. Then \(g\) acts on both
crystalline cohomology and étale cohomology of \(X\).
By~\cite[Theorem~2(2)]{katz-messing:consequence-riemann-hypothesis},
the characteristic polynomial of
\(\mathrm{H}^{i}_{\text{cris}}(g)\) agrees with the characteristic polynomial of
\(\mathrm{H}^{i}_{\text{ét}}(g,\mathbb{Q}_{\ell})\).
\textit{If \(g\) has finite order} in \(\mathrm{Aut}(X)\), then
both operators above are semisimple, hence
\begin{equation*}
\mathrm{H}^{i}_{\text{cris}}(g) = \mathrm{Id} \quad \text{if and only if}\quad
\mathrm{H}^{i}_{\text{ét}}(g,\mathbb{Q}_{\ell}) = \mathrm{Id}.
\end{equation*}
In all the cases that we are interested in, the group \(\mathrm{Aut}(X)\) is
finite. So proving the faithfulness of the action \(\mathrm{Aut}(X)\) on étale
cohomology is the same as proving the faithfulness of the action of
\(\mathrm{Aut}(X)\) on the crystalline cohomology. In the sequel we shall treat
only crystalline cohomology.

\medskip%
Our proof of the faithfulness is based on the following theorem of the
second-named author.

\begin{theorem}[\cite{PANL}, Theorem~1.7]%
\label{theorem:pan-bloch-semi-regularity}
Let \(\pi: \mathcal{X} \to \mathrm{Spec}(W(k))\) be a smooth projective morphism
of pure relative dimension \(n\). The special fiber \(\mathcal{X}\times_{W(k)}
\mathrm{Spec}(k)\) is denoted by \(X\). Let \(g: X \to X\) be an automorphism of
\(X\). Assume that
\begin{enumerate}
\item under the canonical isomorphism
\begin{equation*}
\mathrm{H}_{\textup{cris}}^{n}(X/W(k))  = \mathrm{H}_{\textup{dR}}^{n}(\mathcal{X}/W),
\end{equation*}
\(g^{\ast}\) preserves the Hodge filtration;
\item the infinitesimal Torelli theorem holds for \(X\), i.e., the map
\begin{equation*}
\Psi_0: \mathrm{H}^{1}(X,T_{X}) \to
\bigoplus_{i+j=n}
\mathrm{Hom}(\mathrm{H}^j(X,\Omega_X^i),\mathrm{H}^{j+1}(X,\Omega_X^{i-1})).
\end{equation*}
is injective; and
\item the Hodge-to-de~Rham spectral sequence of \(X\) is \(E_1\){}-degenerate,
  and the Hodge cohomology groups
  \(\mathrm{H}^{i}(\mathcal{X},\Omega^{j}_{\mathcal{X}/W(k)})\) are free
  \(W(k)\){-}modules.
\end{enumerate}
then there exists an automorphism \(\widetilde{g}:\mathcal{X}\to\mathcal{X}\)
over \(\mathrm{Spec}(W(k))\), such that \(\widetilde{g}\) reduces to \(g\)
modulo \(p\).
\end{theorem}

This theorem allows us to deduce the faithfulness of the action of
\(\mathrm{Aut}(X)\) on \(\mathrm{H}_{\text{cris}}^{n}(X/W(k))\) from
corresponding results in characteristic \(0\).

\begin{proposition}%
\label{proposition:automorphism-faithful-lift}
Let \(X\) be a smooth complete intersection over \(k\) satisfying the
infinitesimal the Torelli theorem. Assume that \(X\) is not a quadric nor an
intersection of two quadrics. Then \(\mathrm{Aut}(X)\) acts faithfully on
\(\mathrm{H}^{n}_{\textup{cris}}(X/W(k))\).
\end{proposition}

\begin{proof}
Let \(g\) be an automorphism that acts trivially on
\(\mathrm{H}^{n}_{\text{cris}}(X/W(k))\), then for any lift \(\mathcal{X}\),
\(g^{\ast}\) trivially preserves the Hodge filtration on the crystalline
cohomology induced by the de~Rham cohomology, so the condition (1) above is met.
Since we have assumed that \(X\) satisfies the infinitesimal Torelli theorem in
characteristic \(p\), the condition (2) is met as well. It then follows from
Theorem~\ref{theorem:pan-bloch-semi-regularity} that for any lift
\(\mathcal{X}\), there exists an automorphism \(\widetilde{g}\) lifting \(g\).
It follows that the map
\begin{equation*}
h :=\widetilde{g} \otimes_{W(k)} W(k)[1/p]
: \mathcal{X} \otimes_{W(k)} W(k)[1/p] \to \mathcal{X} \otimes_{W(k)} W(k)[1/p]
\end{equation*}
induces the identity on the de~Rham cohomology
\(\mathrm{H}_{\text{dR}}^{n}(\mathcal{X}/W(k))[1/p]\). Choosing an embedding
\(\sigma: W(k)[1/p] \to \mathbb{C}\), and using the comparison between the algebraic
de~Rham and singular cohomology, we know that the morphism
\begin{equation*}
h^{\sigma} : \mathcal{X}^{\sigma} \to \mathcal{X}^{\sigma}
\end{equation*}
induces an isomorphism on the singular cohomology
\(\mathrm{H}^{n}(X,\mathbb{C})\) with \(\mathbb{C}\){-}coefficients. It follows
that \(\mathrm{H}^{n}(h^{\sigma},\mathbb{Q})\) induces a map for singular
cohomology with \(\mathbb{Q}\){-}coefficients. By
Proposition~\ref{ci-faithful-cy-fano} and Proposition~\ref{ci-faithful-gt}, we
must have \(h^{\sigma}=\mathrm{Id}\), hence \(h=\mathrm{Id}\). Taking
closure, we conclude that \(\widetilde{g}\), hence \(g\), is the identity map.
\end{proof}

By Proposition~\ref{theorem:flenner-ci} in the appendix, and the discussion
following it, besides cubics, which we have dealt with in
Lemma~\ref{lemma:cubic-auto}, cubic fourfolds, which had been considered
in~\cite{pan:automorphism-cohomology-i}, the infinitesimal Torelli theorem holds true for all the rest
complete intersections except for quadrics and (2,2)-type complete
intersections. Thus, Proposition~\ref{proposition:automorphism-faithful-lift}
concludes the proof of Theorem~\ref{autocoh}(ii).

The rest of this section will be dealing with complete intersections of type
(2,2).

\subsection*{\thetheorem.}
\label{sec:2-quadrics}
We shall next consider automorphisms of intersections of two quadrics. In the
sequel, we shall assume that \(k\) is an algebraically closed field whose
characteristic is not equal to $2$. Let \(\ell\) be a prime number different
from the characteristic of \(k\).
Let $X$ be a smooth complete intersection of
type $(2,2)$ in $\mathbb{P}^n_k$. We assume \(m = \dim X = n - 2 \geq 3\).
By~\cite[Proposition 2.1]{R}, $X$ is projectively equivalent to $Q_1\cap Q_2$
with
\begin{equation}\label{eqfortt}
Q_1=V\left(\sum\limits_{i=0}^n X_i^2=0\right),\quad
Q_2=V\left(\sum\limits_{i=0}^n \lambda_i X_i^2=0\right)
\end{equation}
where $\{\lambda_i\}$ are distinct elements in $k$ and \(Q_1\) intersects
\(Q_2\) transversely. By Theorem~\ref{THMCIAUTO000}, we have
\[
\mathrm{Aut}(X)\cong (\mathbb{Z}/2\mathbb{Z}){}^{n+1}/
\{\pm 1\}\cong(\mathbb{Z}/2\mathbb{Z}){}^{n}.
\]
Define \(\sigma_i\) to be the automorphism of \(X\) induced by
\begin{equation*}
[x_0,\ldots,x_n] \mapsto [x_0,\ldots,x_{i-1},-x_i,x_{i+1},\ldots,x_n].
\end{equation*}
Let \(\mathrm{Aut}_{\text{tr}}(X)\) be the kernel of
\begin{equation*}
\mathrm{Aut}(X)\rightarrow \mathrm{Aut}(\mathrm{H}^m_{\text{\'et}}(X,\mathbb{Q}_\ell)).
\end{equation*}

We first consider even dimensional quadrics.

\begin{lemma}%
\label{lemma:even-dimension-intersection-quadric-auto-trivial-trivial}
If $m$ is even, then $\mathrm{Aut}_{\textup{tr}}(X) = \{\mathrm{Id}\}$.
\end{lemma}

\begin{proof}
Let \(m=2e\). Let \(\Sigma\) be the space of \(e\){-}planes contained in \(X\).
Let \(\eta\) be the class in
\(\mathrm{H}^{m}_{\text{\'et}}(X,{\mathbb{Q}}_{\ell})\) defined by the
cohomology class of an intersection of \(X\) with a collection of generic
hyperplanes. Let \(F(X)\) be the free abelian group generated by \(\Sigma\) and
\(\eta\), and let \(A(X)\) be the quotient of \(F(X)\) modulo the relations
\begin{equation*}
\eta - s - s_i - s_j - s_{ij},
\end{equation*}
for \(s \in \Sigma\). Here, \(s_i = \sigma_i(s)\), and \(s_{ij} =\sigma_i \sigma_j(s)\).

Let $g\in \mathrm{Aut}_{\text{tr}}(X)$. By~\cite[Page~48]{R}, we have a
commutative diagram
\[
\begin{tikzcd}
  A(X)\ar[hook,d]\ar[r,"\widehat{g}"] & A(X)\ar[hook,d] \\
  \mathrm{H}^{m}_{\text{\'et}}(X,{\mathbb{Q}}_{\ell})
  \ar[r, "g^\ast=\mathrm{Id}"] &
  \mathrm{H}^{m}_{\text{\'et}}(X,{\mathbb{Q}}_{\ell})
\end{tikzcd}
\]
It follows from the remark \cite[Chapter 3, Page 49]{R} that the
map \[A(X)\rightarrow \mathrm{H}^m(X)\] is injective. In particular, we have
$\widehat{g}=\mathrm{Id}$. By~\cite[Proposition 3.18]{R}, we have an injection
\[
\mathrm{Aut}(X)\hookrightarrow \mathrm{Aut}(\Sigma)
\hookrightarrow \mathrm{Aut}(A(X)).
\]
Therefore, we have $g=\mathrm{Id}$.
\end{proof}

We now deal with odd dimensional quadrics. We first need a lemma.

\begin{lemma}\label{lemmint}
Suppose that $X$ is a smooth complete intersection of type $(2,2)$. If
$\dim(X)=m=2e+1$ is odd, then a general point $p\in X$ is the intersection of
two $e$-planes in $X$. In particular, \(p\) is the intersection of all
\(e\){-}planes containing \(p\).
\end{lemma}

\begin{proof}
Suppose $T_{X,p}$ is the projective tangent space to $X$ at $p$. By a remark in
\cite[Page 65]{R}, the intersection $T_{X,p}\cap X$ is a cone of a nonsingular
intersection $Y$ of two quadrics in $T_{X,p}$ if $p$ is a general point of $X$.
It is clear that $\dim(Y)=m-3$ is even. By \cite[Theorem 3.8]{R}, there are two
$(e-1)$-planes $s$ and $t$ in $Y$ such that $s\cap t$ is empty. Note that the
cones $\mathrm{Cone}(s)$ and $\mathrm{Cone}(t)$ with vertex $p$ are two
$e$-planes in $X$ containing $p$. It is clear that the intersection
$\mathrm{Cone}(s)\cap \mathrm{Cone}(t)$ is the point $p$.
\end{proof}

\begin{lemma}
Let notation be as in~\ref{sec:2-quadrics}. Assume that \(m\) is odd. Then
$\mathrm{Aut}_{\textup{tr}}(X)$ is $(\mathbb{Z}/2\mathbb{Z}){}^{m+1}$. Suppose
an automorphism $g$ of $X$ is given by
$\sigma_{i_1}\circ\ldots\circ\sigma_{i_s}$. Then
$g\in\mathrm{Aut}_{\textup{tr}}(X)$ if and only if $s$ is even.
\end{lemma}

\begin{proof}
Suppose $m=2e+1$. Let $S$ be the variety parameterizing $e$-planes in $X$. By
\cite[Chapter 4, Page 55]{R}, the variety $S$ is nonsingular of dimension $e+1$.
The incidence subvariety $T=\{(s,x)|x\in s\}\subset S\times X$
\[
\begin{tikzcd}
&T\ar[dl,swap,"p_1"] \ar[dr,"p_2"]&\\
S& & X
\end{tikzcd}
\]
induces an isomorphism
$\mathrm{H}_{\text{\'et}}^1(S)\cong\mathrm{H}_{\text{\'et}}^{2e+1}(X)$,
see~\cite[Theorem 4.14]{R}. Moreover loc.~cit.~also says that an isomorphism $g
\in \mathrm{Aut}_{L}(X)$ induces $\widehat{g}:S\rightarrow S$ and a commutative
diagram
\[
\begin{tikzcd}
\mathrm{H}_{\text{\'et}}^1(S,{\mathbb{Q}}_{\ell})\ar[d,"{\widehat{g}^*}"] \ar[r,"{\cong}"] & \mathrm{H}_{\text{\'et}}^{2e+1}(X,{\mathbb{Q}}_{\ell})\ar[d,"{g^*}"]\\
\mathrm{H_{\text{\'et}}}^1(S,{\mathbb{Q}}_{\ell})\ar[r,"{\cong}"] &
\mathrm{H}_{\text{\'et}}^{2e+1}(X,{\mathbb{Q}}_{\ell}).
\end{tikzcd}
\]

We shall need two claims to complete our argument.

\medskip\noindent%
\textbf{Claim~1.} \textit{If $\widehat{g}$ has a fixed point and
  $\widehat{g}^*=\mathrm{Id}$, then $g=\mathrm{Id}$ on $X$}.

\medskip%
In fact,
by~\cite[Theorem 4.8]{R}, we know $S$ is an abelian variety. From the injection
\[
\mathrm{End}(S)\hookrightarrow
\mathrm{End}(\text{H}^1_{\text{\'et}}(S,{\mathbb{Q}}_{\ell})),
\]
it follows that $\widehat{g}=\mathrm{Id}$. On the other hand, by
Lemma~\ref{lemmint}, a general point $p\in X$ is the intersection of $e$-planes
in $X$ which contain $p$. In other words, we have $p=\bigcap P_i$ where $P_i$
runs through all $e$-planes in $X$ containing $p$. It follows that
\[
g(p)=g\left(\bigcap_{p\in P_i} P_i\right)=\bigcap_{p\in P_i}\widehat{g}(
P_i)=\bigcap_{p\in P_i} P_i=p
\]
for a general point $p\in X$. We conclude that $g=\mathrm{Id}$.

\medskip\noindent%
\textbf{Claim~2.}
\(\sigma_i^*=-\mathrm{Id}\) on
\(
\mathrm{H}_{\text{\'et}}^{m}(X,{\mathbb{Q}}_{\ell})\cong
\mathrm{H}_{\text{\'et}}^1(S,{\mathbb{Q}}_{\ell})
\).

\medskip%
In fact, since the automorphism $\sigma_i$ fixes $\{x_i=0\}\cap X$, which is a
smooth complete intersection of type $(2,2)$ of dimension $2e$, it is clear that
$\sigma_i$ fixes all the $e$-planes in $\{X_i=0\}\cap X$. Let $s$ be a $e$-plane
in $\{X_i=0\}\cap X$. It follows that $\widehat{\sigma_i}([s])=[s]$. Let $C$ be
the hyperelliptic curve (cf. \cite[Page 56]{R}) given by:
\[
z^2=\prod_{i=0}^{n} (x-\lambda_i y).
\]
For the point $[s]\in S$, by~\cite[Proposition 4.2]{R}, we have a map
\[
\Psi:
C\rightarrow S
\]
whose image is the closure of $C^{1'}_s=\{t\in S|\dim(s\cap t)=e-1\}$.
Therefore, the automorphism $\widehat{\sigma_i}:S\rightarrow S$ maps $C^{1'}_s$
to $C^{1'}_s$. It induces an automorphism \[g_i:C\rightarrow C\] such that
$\widehat{\sigma_i}\circ \Psi=\Psi\circ g_i $, i.e., the following diagram commutes.
%\[
%\xymatrix{C\ar[r] \ar@/^2pc/[rr]^{\Psi}\ar[d]^{g_i} & Jac(C)\ar[d]^{H} \ar[r]^{F} &S\ar[d]^{\widehat{\sigma_i}}\\
%C\ar[r] &Jac(C)\ar[r]^F &S}
%\]
\[
\begin{tikzcd}
C\ar[r,"{\Psi}"]\ar[d,"{g_i}"] &S\ar[d,"{\widehat{\sigma_i}}"]\\
C\ar[r,"{\Psi}"] & S
\end{tikzcd}
\]
By~\cite[Proposition 4.2, Corollary 4.5, Lemma 4.7 and Theorem 4.8]{R}, we
conclude that the morphism $\Psi$ induces an
isomorphism
\[
\Psi^*:\text{H}_{\text{\'et}}^1(C,\mathbb{Q}_\ell)\rightarrow\text{H}_{\text{\'et}}^1(S,\mathbb{Q}_\ell).
\]
By the main theorem of~\cite{Poonen}, there exist $\lambda_i$, \(i=0,\ldots,n\)
such that
\[
\mathrm{Aut}(C)=\{1,\iota \}
\]
where $\iota$ is the natural involution. In particular, it follows from the
previous claim that $(\widehat{\sigma_i})^*\ne \mathrm{Id}$. We conclude
$g_i\neq \mathrm{Id}$. Therefore, we have $g_i=\iota$. It follows that
$\widehat{\sigma_i}^*=-\mathrm{Id}$ on
$\mathrm{H}_{\text{\'et}}^1(S,{\mathbb{Q}}_{\ell})=\mathrm{H}_{\text{\'et}}^{m}(X,{\mathbb{Q}}_{\ell})$.
So we have shown the claim holds for one particular intersection
$X=X_{\{\lambda_i\}}$ of two quadrics which is determined by a certain
$\{\lambda_i\}_{i=0}^n$.

The general case is handled as follows. We consider the smooth family of
complete intersections of two quadrics
\[
\xymatrix{X_{\{\lambda_i\}}\ar[d] \ar@{}[r]|-{\subseteq}  & \widehat{X}=\{\sum X_i^2=0,\sum \lambda_i X_i^2=0\}\ar[d]_f\\
  \{\lambda_i\} \ar@{}[r]|-{\in}& B }
\]
The automorphism $\sigma_i$ acts on $\widehat{X}$ fiberwise. Thus
\(\sigma_i^{\ast}+\mathrm{Id}\in\mathrm{End}(\mathrm{R}^{m}f_{\ast}{\mathbb{Q}}_{\ell})\),
hence it is a global section of the lisse-\'etale sheaf
\(\mathit{End}(\mathrm{R}^{m}f_{\ast}{\mathbb{Q}}_{\ell})\). Its vanishing at
\(X_{\{\lambda_i\}}\) shows that it vanishes everywhere. This proves the claim.

\medskip%
Combining the two claims, we conclude that for an automorphism
$g=\sigma_{i_1}\circ\ldots\circ\sigma_{i_s}$ of $X$, $g\in
\mathrm{Aut}_{\mathrm{tr}}(X)$ if and only if $s$ is even.
\end{proof}

\appendix

\section{The infinitesimal Torelli theorem over an arbitrary field}\label{app}

We give an exposition on infinitesimal Torelli theorem of
complete intersections, a precise statement is given in
Theorem~\ref{theorem:pairing-injectiv} and Theorem~\ref{theorem:flenner-ci}. We
present these notes because we need to apply a positive characteristic version
of Theorem~\ref{theorem:flenner-ci} to deduce the faithfulness of automorphism
actions on cohomology groups of certain complete intersections.

Flenner shows the infinitesimal Torelli theorem holds for certain complex
manifolds ~\cite{Inf}. His method can also be used to prove infinitesimal
Torelli theorems for algebraic varieties over an arbitrary field as long as one
takes a bit care with the duality between symmetric powers. Although the proof
we exposed in Section~\ref{sec:thes-spectr-sequ} looks slightly different from
Flenner's original proof (we find our spectral sequence argument makes the
book-keeping less painful), the idea of using the resolutions $T$, $A$, $B$ (see
Section~\ref{sec:three-resolutions}), certainly goes back to Flenner.
Of course, we should not claim any originality here.

\refstepcounter{thm}
\subsection*{\thetheorem.~Divided powers}
Recall that the $r$th \emph{divided power} of a locally free sheaf $E$ on a
scheme, notation $D_r(E)$, is defined to be the dual
\begin{equation}
\label{eq:divided-power}
D_r(E) = \mathrm{Sym}^r(E^*){}^*.
\end{equation}
If $e_1,\ldots,e_n$ is a local frame of $E$, then we write
\begin{equation*}
e_1^{(i_1)}\cdots e_n^{(i_n)}
\end{equation*}
to be the element dual to the basis elements
\begin{equation*}
(e_1^*){}^{i_1}\cdots (e_n^*){}^{i_n}.
\end{equation*}
If $u = \sum u_i e_i$, one defines
\begin{equation}
\label{eq:divided-power-of-a-vector}
u^{(r)} = \sum_{p_1+\cdots p_n= r} u_1^{p_1}\cdots u_n^{p_n} e_1^{(p_1)}\cdots e_n^{(p_n)}.
\end{equation}
This definition does not depend on the choice of basis. There is a natural
pairing between divided power and symmetric power:
\begin{equation}
\label{eq:sym-div}
\mathrm{Sym}^{p+s}(E) \otimes D_s(E^*) \to \mathrm{Sym}^p(E),
\end{equation}
dual to the algebra structure on $D_*(E^*)$.

The bialgebra structure on the symmetric algebra $\mathrm{Sym}^*(E)$ defines a
bialgebra structure on $D_*(E)$.

\begin{situation}
\label{situation:flenner}
Let $X$ be a smooth, proper scheme pure dimension $n$ over a field $k$. Assume
that there exists a short exact sequence of locally free sheaves
\begin{equation*}
0 \to \mathcal{G} \to \mathcal{F} \to \Omega_{X}^1 \to 0.
\end{equation*}
Let $p$ be a positive integer no larger than $n$. We make the following two
hypotheses:
\begin{enumerate}
\item[(i)] the map
\begin{equation*}
\mathrm{H}^0(X,D_{n-p}(\mathcal{G}^*)\otimes \omega_X)\otimes
\mathrm{H}^0(X,D_{p-1}(\mathcal{G}^*)\otimes \omega_X) \to
\mathrm{H}^0(X,D_{n-1}(\mathcal{G}^*)\otimes \omega_X^2)
\end{equation*}
is surjective;
\item[(ii)] for all $0 \leq j \leq n-2$, we have
\begin{equation*}
\textstyle\mathrm{H}^{j+1}
(X, \mathrm{Sym}^j(\mathcal{G})\otimes \bigwedge^{n-1-j}(\mathcal{F})\otimes \omega_X^{-1})
 = 0.
\end{equation*}
\end{enumerate}
\end{situation}

We have a long exact sequence of locally free sheaves
\begin{equation}
\label{eq:resolving-exterior-omega}
0 \to \mathrm{Sym}^{p-1}(\mathcal{G}) \to \cdots \to \mathcal{G} \otimes
\textstyle\bigwedge^{p-2}(\mathcal{F}) \to \bigwedge^{p-1}(\mathcal{F}) \to \Omega_X^{p-1}
\to 0.
\end{equation}
Dualizing~(\ref{eq:resolving-exterior-omega}), replacing $p$ by $p+1$, we get
\begin{equation}
\label{eq:dualize-resolving-exterior-omega}
(\Omega_X^p){}^* \to \textstyle\bigwedge^p(\mathcal{F}^*) \to \mathcal{G}^*\otimes
\bigwedge^{p-1}(\mathcal{F}^*) \to \cdots \to D_p(\mathcal{G}^*)\to 0.
\end{equation}
Since the pairing
\begin{equation*}
\Omega^p_X\otimes \Omega_X^{n-p} \to \omega_X
\end{equation*}
is a perfect pairing, tensoring $\omega_X$ to
(\ref{eq:dualize-resolving-exterior-omega}) and change $p$ by $n-p$ yields
\begin{equation}
\label{eq:coresolving-omega-p}
0 \to\Omega_X^p \to \textstyle\bigwedge^{n-p}(\mathcal{F}^*) \otimes \omega_X \to \cdots
\to D_{n-p}(\mathcal{G}^*)\otimes\omega_X \to 0.
\end{equation}
When $p=1$, dualize (\ref{eq:coresolving-omega-p}) yields
\begin{equation}
\label{eq:resolving-tangent}
0 \to \omega_X^*\otimes \mathrm{Sym}^{n-1}(\mathcal{G})\to\cdots\to
\omega_X^*\otimes \textstyle\bigwedge^{n-1}\mathcal{F}  \to T_X \to 0.
\end{equation}

\refstepcounter{thm}
\subsection*{\thetheorem.~Three resolutions}
\label{sec:three-resolutions}
Following Flenner, we introduce three complexes $T$, $A$ and $B$ of locally free
sheaves that resolve the sheaves $T_X$, $\Omega^p_X$ and $\Omega_X^{p-1}$
respectively. Let notation be as in Situation~\ref{situation:flenner}. Define
\begin{equation}
T_\bullet = \{ \omega_X^*\otimes \mathrm{Sym}^{n-1}(\mathcal{G})\to\cdots\to
\omega_X^*\otimes \textstyle\bigwedge^{n-1}\mathcal{F}\}
\end{equation}
placed at cohomological degree $-(n-1), \ldots, -1, 0$;
\begin{equation}
A^\bullet = \{\textstyle\bigwedge^{n-p}(\mathcal{F}^*) \otimes \omega_X \to \cdots
\to D_{n-p}(\mathcal{G}^*)\otimes\omega_X \}
\end{equation}
placed at cohomological degree $0, 1, \ldots,n-p$; and
\begin{equation}
B_\bullet = \{\mathrm{Sym}^{p-1}(\mathcal{G}) \to \cdots \to \mathcal{G} \otimes
\textstyle\bigwedge^{p-2}(\mathcal{F}) \to \bigwedge^{p-1}(\mathcal{F}) \}
\end{equation}
placed at cohomological degree $-(p-1),\ldots, -1,0$.
Then we have quasi-isomorphisms
\begin{equation*}
T \to T_X[0],\quad \Omega_X^p[0] \to A,\quad B \to \Omega_X^{p-q}[0].
\end{equation*}
The contraction
\begin{equation*}
T_X \otimes \Omega_X^p \to \Omega^{p-1}_X
\end{equation*}
thus induces a morphism
\begin{equation}
\label{eq:11}
\langle \cdot , \cdot \rangle : T \otimes A \to B,
\end{equation}
in the derived category $\mathrm{D^b_{Coh}}(X)$. The morphism
$\langle{.,.}\rangle$ has an explicit description: By mulitilinear algebra, there
are natural contractions
\begin{equation*}
\textstyle\bigwedge^{s+t}
(\mathcal{E}) \otimes \bigwedge^s(\mathcal{E}^*) \to \bigwedge^t(\mathcal{E}),
\quad
\mathrm{Sym}^{s+t}(\mathcal{E}) \otimes D_s(\mathcal{E}^*) \to \mathrm{Sym}^t(\mathcal{E}).
\end{equation*}
Note that the hypothesis~\ref{situation:flenner}%
(ii) says precisely that
\begin{equation*}
\mathrm{H}^{j+1}(X,T_j) = 0, \quad 0 \leq j \leq n-2.
\end{equation*}

\refstepcounter{thm}
\subsection*{\thetheorem.~Spectral sequences}
\label{sec:thes-spectr-sequ}
We shall prove the infinitesimal Torelli theorem by an induction argument with
the help of certain spectral sequences. There are spectral sequences
\begin{equation*}
{}^{T}E^{-i,j}_1 = \mathrm{H}^j(X,T_{i}),
{}^{A}E^{i,j}_1 = \mathrm{H}^j(X,A^i),
{}^{B}E^{-i,j}_1 = \mathrm{H}^j(X,B_i)
\end{equation*}
abutting to $\mathrm{H}^{*}(X,T_X)$, $\mathrm{H}^{*}(X,\Omega_X^{p})$ and
$\mathrm{H}^{*}(X,\Omega_X^{p-1})$ respectively. The map~(\ref{eq:11}) thus
induces a morphism
\begin{equation*}
\langle{\cdot,\cdot}\rangle_m : {}^T E_m^{-i,j} \otimes {}^A E_m^{s,t} \to {}^B E_m^{s-i,j+t}.
\end{equation*}
on the level of spectral sequences (i.e., compatible with differentials). This
follows from the construction of the spectral sequences (since the naïve
filtration, which is used to define these spectral sequences is preserved under
the tensor product).

Ultimately we are interested in the map
\begin{equation}
\label{eq:2}
\langle{\cdot,\cdot}\rangle :
\mathrm{H}^1(X,T_X) \otimes \mathrm{H}^{n-p}(X,\Omega_X^p) \to \mathrm{H}^{n-p+1}(X,\Omega_X^{p-1}).
\end{equation}
The product structure on the spectral sequences eventually computes the tensor
product of assoicated graded vector spaces for the infinitesimal Torelli map.

We wish to prove that~\eqref{eq:2} is injective on the first factor, in the
sense that if an element $\xi \in \mathrm{H}^1(X,T_X)$ induces the zero linear
function, i.e., $\langle{\xi,\cdot}\rangle = 0$, then $\xi$ must be zero. The
relevant terms in ${}^T E_m^{-i,j}$ are those with $j - i = 1$.

Hypothesis~\ref{situation:flenner}%
(ii) says that ${}^T E_{1}^{-i,j} = 0$ for all $j - i = 1$ and $i \neq -(n-1)$.
Since $(-(n-1),n)$ is the left highest edge of the spectral sequence, we infer
that $\mathrm{H}^1(X, T_X) = {}^T E_\infty^{-(n-1),n}$ and is the intersection
\[
\bigcap_{m} \mathrm{Ker}(d_m^{-(n-1),n}) \subset {}^T E_1^{-(n-1),n}.
\]
We shall denote by $V_{n-1-m}$ the subspace $\mathrm{Ker}(d_m^{-(n-1),n})$ of
$V_{n-1} = {}^T E_{1}^{-(n-1),n}$. Then $V_m \subset V_{m-1}$ and
$\mathrm{H}^1(X,T_X)$ is $V_{0}$. It turns out Flenner's condition (i) of
Hypothesis~\ref{situation:flenner} puts some non-deneneracy on the first stage of
the spectral sequences.

\begin{lemma}
\label{lemma:step1}
Let notation be as above. The pairing
\begin{equation*}
\langle{\cdot,\cdot}\rangle : V_0 \otimes {}^A E_1^{n-p,0} \to
{}^B E_1^{-(p-1),n}
\end{equation*}
is injective on the first factor.
\end{lemma}

\begin{proof}
By the definition of the spectral sequence, the pairing is
\begin{equation*}
\mathrm{H}^n(X,\omega_X^*\otimes \mathrm{Sym}^{n-1}(\mathcal{G}))
\otimes
\mathrm{H}^{0}(X,D_{n-p}(\mathcal{G}^*)\otimes \omega_X) \to
\mathrm{H}^n(X,\mathrm{Sym}^{p-1}(\mathcal{G})).
\end{equation*}
But saying the injectivity on the first factor amounts to saying, thanks to the
Serre duality, that the pairing in Situation~\ref{situation:flenner} (i) is
surjective. We win.
\end{proof}

\begin{lemma}
\label{lemma:step2}
For all $m>0$, we have
\begin{equation*}
V_{n-m} \otimes {}^A E_m^{n-p,0} \to {}^B E_m^{-(p-1),n}
\end{equation*}
is injective on the first factor.
\end{lemma}

\begin{proof}
We have the following diagram
\begin{equation*}
\begin{tikzcd}[column sep = -0.5em]
V_{n-m+1} & \otimes & {}^A E_{m-1}^{n-p,0} \ar[two heads]{d}  \ar{rr} &\qquad  & {}^B E_{m-1}^{-(p-1),n} \\
V_{n-m} \ar[hook]{u}& \otimes & {}^A E_{m}^{n-p,0} \ar{rr} &\qquad & {}^B E_{m}^{-(p-1),n} \ar[hook]{u}
\end{tikzcd}
\end{equation*}
The right vertical arrow in injective since the coordinate $(-(p-1),n)$ of the
spot forces so (it's the left-most highest spot for ${}^B E$). By the definition
of $V_{n-m}$, an element $v$ of it satisfies the property
\begin{equation*}
\langle{v,d_{m-1}a} \rangle = 0.
\end{equation*}
Since ${}^A E_m^{n-p,0}$ is defined to be the quotient
${}^A E_{m-1}^{n-p,0}/\mathrm{Im}(d_{m-1})$, we see the second horizontal
pairing is indeed injective on the first fact, by induction on $m$.
\end{proof}

\begin{theorem}
\label{theorem:pairing-injectiv}
Let notation be as in Situation~\ref{situation:flenner}. The pairing
\begin{equation*}
\mathrm{H}^1(X,T_X) \otimes \mathrm{H}^{n-p}(X,\Omega_X^p)
\to
\mathrm{H}^{n-p+1}(X,\Omega^{p-1}_X)
\end{equation*}
is injective on the first factor.
\end{theorem}

\begin{proof}
Let $\xi$ be an element in $E_\infty^{-(n-1),n} = V_0 = \mathrm{H}^1(X,T_X)$
that pairs to zero with the second factor. Then it kills all elements in
${}^AE_{\infty}^{n-p,0}$ (the right lowest edge of the spectral sequence) since this
is a subspace of $\mathrm{H}^{n-p}(X,\Omega^p_X)$. By Lemma~\ref{lemma:step2},
$\xi$ is zero. This completes the proof.
\end{proof}

\refstepcounter{thm}
\subsection*{\thetheorem.~Applications to complete intersections}
Flenner's theorem has several applications. Most notably is the validity of the
infinitesimal Torelli for complete intersections in projective spaces. Using
Theorem~\ref{theorem:pairing-injectiv}, Flenner, in his original paper, proves
the infinitesimal Torelli theorem for complete intersections by verifying the
his conditions (i) and (ii) for them. The proof is about linear algebra of
polynomials, and only uses the Bott vanishing theorem for the projective space
as an extra input, hence does not depend on the ambient field. The only possible issue
is the duality between symmetric power and divided power, but this is readily
resolved by noticing the following fact: If $L_1,\ldots,L_r$ are line bundles on
a variety $X$, then there is an isomorphism of \emph{bialgebras}
\[
D_*(L_1\oplus \cdots \oplus L_r) \cong \bigoplus L_1^{i_1} \otimes \cdots \otimes L_r^{i_r}.
\]
This is true, since the dual algebra of $D_*$ is the symmetric algebra, and
the symmetric algebra version of this isomorphism is well-known.

\begin{proposition}%
\label{theorem:flenner-ci}
Let $k$ be any algebraically closed field. Let $X$ be a complete intersection of
type $(d_1,\ldots,d_c)$ (\(d_i \geq 2\)) in a projective space $\mathbb{P}^{n+c}$. Then the
conditions \textup{(i), (ii)} of Hypothesis~\ref{situation:flenner} are verified
for the exact sequence
\begin{equation*}
0 \to \bigoplus_{i=1}^c \mathcal{O}_X(-d_i) \to \Omega_{\mathbb{P}^{n+c}}^1|_X \to \Omega_X^1 \to 0.
\end{equation*}
except the cases listed below. Hence the infinitesimal Torelli theorem holds for
all complete intersections except for the ones listed below.
\end{proposition}

There are a few cases that (i) and (ii) do not verify, these exceptions are:
\begin{enumerate}
\item (2,2) complete intersections,
\item cubic surfaces,
\item even dimensional (2,3) complete intersections,
\item even dimensional (2,2,2) complete intersections,
\item cubic fourfolds,
\item quartic K3 surfaces, and
\item quadrics.
\end{enumerate}
As is well-known, the infinitesimal Torelli fails for cubics and even
dimensional (2,2) complete intersections. For quadrics the infinitesimal
Torellis theorem is trivial.

We need the infinitesimal Torelli for a complete intersection over a field of
characteristic \(p>0\) in Section~\ref{sec:aut-and-coh} to conclude the proof of
Theorem~\ref{autocoh}(ii). For cubic fourfolds, the second-named
author has verified the validity of Theorem~\ref{autocoh}(ii) in another paper~\cite{pan:automorphism-cohomology-i}.
Thus, to conclude the proof we need, in addition to the
cases covered by Proposition~\ref{theorem:flenner-ci}, we still need verify the
infinitesimal Torelli holds true for (2,2,2)-type complete
intersections and (2,3)-type complete intersections of even dimension.
Flenner, in his original paper,
already proves these cases by other means (modifying his scheme of proofs). But
we decide to provide an exposition to these facts for the sake of completeness.
Both cases are proven similarly. Flenner treated the case (2,2,2); so we shall
illustrate our method by dealing with (2,3) complete intersections in
$\mathbb{P}^{2p+2}$.

Recall that the Bott vanishing theorem asserts that the cohomology groups
\begin{equation*}
\mathrm{H}^{b}(\mathbb{P}^r,\Omega_{\mathbb{P}^r}^a(\ell))
\end{equation*}
vanish except in the following cases:
\begin{itemize}
\item $b = a$, $\ell = 0$,
\item $b = 0$, $\ell > a$, or
\item $b = r$, $\ell < a - r$.
\end{itemize}
For a smooth (2,3) complete intersection $X$ in $\mathbb{P} = \mathbb{P}^{2p+2}$,
there might be one extra term in ${}^T E_1$ that is nonzero, namely
${}^T E^{-(p-1),p}_1$. This term is
\begin{equation*}
\mathrm{H}^{p}(X, \Omega_{\mathbb{P}}^p|_X \otimes \mathrm{Sym}^{p-1}(\mathcal{O}_X(-2)
\otimes \mathcal{O}_X(-3)) \otimes \mathcal{O}_X(2p-2)).
\end{equation*}
By the binomial theorem, this cohomology group can be decomposed into a direct sum
\begin{equation*}
\bigoplus_{i=0}^{p-1} \mathrm{H}^p(X, \Omega^p_{\mathbb{P}}(-i)|_X){}^{\oplus\binom{p-1}{i}}
\end{equation*}
By the Bott vanishing theorem and the spectral sequence associated to the Koszul
complex, one sees that the only nontrivial summand in the decomposition is
\begin{equation*}
k \cong \mathrm{H}^p(X,\Omega_{\mathbb{P}}^p|_X).
\end{equation*}
If this 1-dimensional space is killed by some differential of ${}^T E$, then it
won't contribute anything to the eventual infinitesimal Torelli map, and we
already win. If the contrary holds, and suppose $\xi \neq 0$ kills all the space
$\mathrm{H}^p(X,\Omega_X^p)$. Let's derive a contradiction. Then each element
$\eta \in \mathrm{H}^p(X,\Omega_X^p) = \mathrm{H}^0(X,\Omega_X^p[p])$ induces a
homomorphism of spectral sequences
\begin{equation*}
\langle{\cdot, \eta}\rangle : {}^T E^{a,b}_m \to {}^B E_m^{a,b+p}.
\end{equation*}
At the $E_1$-stage, by identifying ${}^B E_1^{-(p-1),2p}$ with
$\mathrm{H}^{2p}(X,\omega_X)$ and ${}^T E_1^{-(p-1),p}$ with
$\mathrm{H}^{p}(X,\Omega_{\mathbb{P}}^p|_X)$, this map is identified with the
cup-product map
\begin{equation}
\label{eq:3}
\cup \eta: \mathrm{H}^p(X,\Omega_{\mathbb{P}}^p|_X) \to \mathrm{H}^{2p}(X,\omega_X).
\end{equation}
Since ${}^T E_1^{-(p-1),p} = {}^T E_\infty^{-(p-1),p}$, the image of the above
map actually falls in ${}^B E_\infty^{-(p-1),2p} \subset
\mathrm{H}^{2p}(X,\omega_X)$. The map~\eqref{eq:3} is not always zero as the
pairing
\begin{equation*}
\mathrm{H}^p(X,\Omega_{\mathbb{P}}^p|_X)\otimes \mathrm{H}^p(X,\Omega^{p}_X) \to \mathrm{H}^{2p}(X,\omega_X).
\end{equation*}
is not identically zero (which also yields $\mathrm{H}^{2p}(X,\omega_X)$ survives
in ${}^B E_\infty$).

\end{document}